\documentclass[latterpaper, 10pt]{article}
\usepackage{amsmath}
\usepackage{amsfonts}
\usepackage{amssymb}
\usepackage{amsthm}
\usepackage{verbatim}
\usepackage{mathrsfs}
\usepackage{graphicx}
\numberwithin{equation}{section}
\input xy
\xyoption{all}
\makeindex

\theoremstyle{plain}
\newtheorem{thm}{Theorem}[section]
\newtheorem{prop}[thm]{Proposition}
\newtheorem{cor}[thm]{Corollary}
\newtheorem{lem}[thm]{Lemma}

\theoremstyle{remark}
\newtheorem{rem}[thm]{Remark}

\theoremstyle{definition}
\newtheorem{mydef}[thm]{Definition}
\newtheorem{example}[thm]{Example}

\makeatletter
\@addtoreset{thm}{chapter}
\makeatother

\newcommand{\C}{\mathbb{C}}
\newcommand{\A}{\mathbb{A}}
\newcommand{\Ps}{\mathbb{P}}
\newcommand{\Z}{\mathbb{Z}}
\newcommand{\R}{\mathbb{R}}
\newcommand{\N}{\mathbb{N}}
\newcommand{\Oc}{\mathcal{O}}

\DeclareMathOperator{\HS}{HS}
\DeclareMathOperator{\HSG}{HSG}
\DeclareMathOperator{\spec}{Spec}
\DeclareMathOperator{\sheafspec}{\mathbf{Spec}}

\DeclareMathOperator{\Hom}{Hom}
\DeclareMathOperator{\J}{J}
\DeclareMathOperator{\JG}{JG}
\DeclareMathOperator{\mult}{mult}
\DeclareMathOperator{\character}{char}
\DeclareMathOperator{\Der}{Der}
\DeclareMathOperator{\ord}{ord}
\DeclareMathOperator{\Sym}{Sym}

\title{Logarithmic Jet Spaces and Intersection Multiplicities}
\date{May 14th, 2009}
\author{Seth Clayton Dutter}

\begin{document}
\maketitle

\begin{abstract}
The theory of relative logarithmic jet spaces is developed for log schemes. With this theory the existence of bounds of intersection multiplicities of curves and divisors on certain log schemes is established. This result extends those of Noguchi and Winkelmann in \cite{nogu1} by replacing the semi-abelian condition by a differential one.
\end{abstract}

\vspace{0.5 in}

The goal of this paper is twofold. First and foremost the theory of relative logarithmic jet spaces will be established for log schemes. This approach will be purely algebraic and not rely on the complex topology and derivative as past papers have. Rather the exposition will follow that of Vojta \cite{vojt1}. With this theory developed, the author translates a result of Noguchi and Winkelmann  on semi-abelian varieties over function fields to slightly more general schemes. The full strength of the theory of Log Jet Spaces will not be employed, but this will serve as a tool demonstrating the usefulness of compactification, log geometry and jet spaces.

\section{Mason's Theorem and Generalizations}

In order to introduce Mason's Theorem and its corollaries some preliminary definitions are needed.

\begin{mydef}
Let $f\in \C[z]$ with $f\neq 0$. Define the \textit{conductor} of $f$, denoted $N(f)$, to be the number of distinct zeroes over the complex numbers. Additionally define the \textit{multiplicity} at $p\in \C$, denoted $\mult_p f$, to be the largest integer $n$ such that $(z-p)^n|f$.
\end{mydef}

\begin{example}
Let $f = z^5(z+1)^{11}$, then $N(f) = 2$ since $f$ only has $2$ distinct roots. On the other hand $\mult_0 f = 5$ and $\mult_{-1} f = 11$.
\end{example}

\begin{thm}[Mason] Let $f, g, h\in \C[z]$ be relatively prime polynomials over the complex numbers, not all constant, satisfying $f+g+h=0$. Then
\[
\max \deg\{f,g,h\} \leq N(fgh) - 1.
\]
\end{thm}

\begin{proof}
See \cite{lang93} Theorem IV.7.1.
\end{proof}

The goal of this paper will not be to generalize Mason's Theorem directly, but rather to generalize the following corollary of Mason's Theorem.

\begin{cor}\label{cor:Mason}
Let $f, g\in \C[z]$ be relatively prime polynomials over the complex numbers, not both constant. Then $\mult_p(f+g) \leq N(fg)$ for all $p\in \C$.
\end{cor}

\begin{proof}
Let $p\in \C$ be arbitrary. Then we can write $f+g = (z-p)^nh$, where $h\in \C[z]$ and $n=\mult_p(f+g)$. By Mason's Theorem,
\[
\begin{split}
\deg{(f+g)} &\leq N(fgh(z-p))-1\\
&\leq N(fg)+N(h).
\end{split}
\]
However, $\deg(f+g) = \deg(h) + n$ and $\deg(h) \leq N(h)$. Combining everything gives $\mult_p(f+g)\leq N(fg)$ for every $p\in \C$.
\end{proof}

There are a variety of ways such a statement can be generalized. For one thing function fields other than $\C[z]$ could be considered. Another option would be increasing the number of polynomials considered. However, the approach the author will be taking in this paper is to give this theorem a geometric interpretation and then generalize. A complete proof is given, although it follows almost immediately from the previous corollary. All of the ideas used come from various proofs of Mason's Theorem.

\begin{prop}\label{Mason} Let $\alpha$ be a regular function on $\A_\C^n = \spec \C[x_1,\ldots,x_n]$ and define $X = D(\alpha)$, the open subset whose complement is the support of the divisor associated to $\alpha$. Let $f$ and $g$ be regular functions on $\A_\C^n$ that restrict to units on $X$. Define the divisor $D = (f+g)$. Let $C$ be $\A^1_\C = \spec \C[z]$ minus $N$ points. Then for every morphism $j:C\rightarrow X$ either $j(C) \subset D$ or $\ord_p j^*(D) \leq N$ for all $p\in C$.
\end{prop}

\begin{proof}
Let $j^\#: \Gamma(X,\Oc_X)\rightarrow \Gamma(C,\Oc_C)$ be the morphism of rings corresponding to $j$. Since $f,g$ are units on $\Oc_X$, $j^\#(f)$ and $j^\#(g)$ must be units on $\Oc_C$. If $j^\#(f+g) = 0$ then $j(C) \subset D$ and we are done. Suppose $j^\#(f+g) \neq 0$, and denote by $f_C, g_C, h_C$ the elements $j^\#(f), j^\#(g), j^\#(f+g)$ respectively giving $f_C+g_C=h_C$. Rewrite the equation as
\[
\frac{f_C}{g_C}+1 = \frac{h_C}{g_C}.
\]
Applying the complex derivative to both sides gives
\begin{equation}\label{prop:masoneq1}
\frac{f_C}{g_C}\left(\frac{f_C'}{f_C}-\frac{g_C'}{g_C}\right) = \frac{h_C}{g_C}\left(\frac{h_C'}{h_C}-\frac{g_C'}{g_C}\right).
\end{equation}
Note that $f_C'/f_C$ is just the log derivative of $f_C$, and as such it can be written as a sum
\[
\frac{f'_C}{f_C} = \sum_{i=1}^N \frac{a_i}{z-\beta_i}.
\]
Likewise $g_C'/g_C$ can be written as a sum
\[
\frac{g'_C}{g_C} = \sum_{i=1}^N \frac{b_i}{z-\beta_i},
\]
where $z-\beta_i\in \Oc_C^*$. Finding a common denominator we can rewrite $f'_C/f_C-g'_C/g_C = r/s$ where $r,s\in \C[z]$ and have no common factors. By construction $\deg(r) \leq N-1$, $\deg(s) \leq N$ and $s\in \Gamma(C,\Oc_C)^*$. Similarly we can write $h'_C/h_C-g'_C/g_C = u/v$. By substituting into Equation \ref{prop:masoneq1} we get
\[
\frac{rf_C}{sg_C} = \frac{uh_C}{vg_C}.
\]
For any point $p\in C$ the polynomials $g_C, f_C, s$ do not vanish since they are units. Therefore
\[
\ord_p \frac{rf_C}{sg_C} = \ord_p r \leq N-1.
\]
On the right hand side we have
\[
\ord_p \frac{uh_C}{vg_C} \geq \ord_p h_C + \ord_p u-1 \geq \ord_p h_C-1.
\]
Combining everything we get $\ord_p h_C \leq N$.
\end{proof}

The simplest application of this proposition is in dimension $1$. In this case it is merely the geometric interpretation of Corollary \ref{cor:Mason}.

\begin{example}
Let $f, g$ be distinct regular functions on $\A^1_\C$, $\alpha = fg$. Observe that $(\A^1_C)_\alpha$ is just the affine line minus $N(fg)$ points. Consider the identity map $(\A^1_\C)_\alpha\rightarrow (\A^1_\C)_\alpha$. By Proposition \ref{Mason} $\mult_p(f(z)-g(z))\leq N(fg)$ for every $p$ such that $\alpha(p)\neq 0$. An alternative interpretation is that the Taylor series centered at any complex number of any two relatively prime distinct polynomials, $f$ and $g$, agree to at most $N(fg)$ terms. This result is in general sharp. Indeed let $f = z^n+1$ and $g = 1$, $N(fg) = n$ and $f-g = z^n$ which has multiplicity $N(fg)$ at $z=0$.
\end{example}

In dimension $1$ there are not many interesting morphisms to choose from. However, in dimension $2$ there is some more freedom.

\begin{example}
Let $\A^2_\C = \spec{\C[x,y]}$, $\alpha = xy$, $D = (y-x)$. Again in this case we can see that the results are sharp. Let $C = (\A^1_\C)_{z(z^N-(z-1)^N)}$ for some non-negative integer $N$ be the affine line minus $N$ points. Consider the morphism $j:C\rightarrow (\A^2_C)_\alpha$ defined by $x\mapsto z^N-(z-1)^N$ and $y\mapsto z^N$, so $y-x\mapsto (t-1)^N$, which has multiplicity $N$ at the point $t=1$.
\end{example}

Part of the weakness of Proposition \ref{Mason} is that it only applies to Cartier divisors which can be represented as a sum of two non-vanishing regular functions. Consider for instance the elliptic curve $y^2=x^3+1$ in $\A^2_\C$. There are a number of different $\alpha$'s that could be selected which would allow the equation $y^2-x^3-1$ to be written as sum of two units. However, $\alpha = xy$, what might be the most natural choice, requires us to treat this as a sum of three units. One particular reason that $xy$ is a natural choice is that $(\A^2_\C)_{xy}\cong \mathbb{G}_m^2$. Now let $C$ be a non-empty open subset of $\A^1_\C$ and $D$ the Cartier divisor associated to the regular function $y^2-x^3-1$ on $(\A^2_\C)_{xy}$. We can ask whether there exists a natural number $N$ such that for every non-constant morphism $j: C\rightarrow (\A^2_\C)_{xy}$, $\ord_p j^*(D)\leq N$ for all $p\in C$. In this particular instance we can view the order at the point $p$ as the degree to which $j(C)$ approximates the curve $y^2-x^3-1=0$ at $j(p)$. If such an $N$ exists then there is a bound on the extent to which the elliptic curve can be approximated by $C$. Unfortunately, Proposition \ref{Mason} is not strong enough to answer whether such a bound exists. Noguchi and Winkelmann answer affirmatively in \cite{nogu1} with the following theorem.\footnote{We have slightly paraphrased the main theorem of \cite{nogu1}.}

\begin{thm}\label{Noguchi} Let $A$ be a semi-abelian variety, let $A\hookrightarrow\bar{A}$ be a smooth equivariant algebraic compactification, let $\bar{D}$ be an effective reduced ample divisor on $\bar{A}$, let $D = \bar{D}\cap A$, and let $C$ be a smooth algebraic curve with smooth compactification $C\hookrightarrow\bar{C}$. Then there exist a number $N\in \N$ such that for every morphism $f:C\rightarrow A$ either $f(C)\subset D$ or $\mult_x f^*D\leq N$ for all $x\in C$. Furthermore, the number $N$ depends only on the numerical data involved as follows:
\begin{enumerate}
\item[i.] The genus of $\bar{C}$ and the number $\#(\bar{C}\setminus C)$ of the boundary points of $C$,
\item[ii.] the dimension of $A$,
\item[iii.] the toric variety (or, equivalently, the associated ``fan") which occurs as the closure in $\bar{A}$ of the maximal connected linear algebraic subgroup $\mathbb{G}_m^t$ of $A$,
\item[iv.] all intersection numbers of the form $D^h\cdot B_{i_1}\cdots B_{i_k}$, where the $B_{i_j}$ are closures of $A$-orbits in $\bar{A}$ of dimension $n_j$ and $h+\sum_j n_j = \dim A$.
\end{enumerate}
\end{thm}

Returning to the previous question, let $C$ be a nonempty open subset of $\A^1_\C$. Then the degree to which $C$ can approximate the elliptic curve $y^2=x^3+1$ in $\mathbb{G}_m^2$ is bounded. After fixing $\Ps^2_\C$ as the compactification of $\mathbb{G}_m^2$, Theorem \ref{Noguchi} implies that the bound for approximating this particular elliptic curve is solely determined by $\#(\Ps^1_\C\setminus C)$.

While Theorem \ref{Noguchi} is remarkably general, there are ways in which it can be further generalized. For one, both Theorem \ref{Noguchi} and Proposition \ref{Mason} only give intersection multiplicities for curves. It turns out that this restriction is not necessary. All we need is a smooth scheme with some other very mild conditions.

The second restriction which is specific to Theorem \ref{Noguchi} is the requirement that $A$ be a semi-abelian variety. In their proof Noguchi and Winkelmann rely on the fact that $A$ is a semi-abelian variety and therefore a Lie group. It turns out that this also is not necessary for the theorem to hold. Not only does $A$ not have to be a semi-abelian variety, it does not even have to be smooth. Certainly some conditions must be in place as the theorem is false for arbitrary schemes. Consider the following example for instance.

\begin{example}
Let $A=C=\spec{\C[z]}$, $D=(z)$ and define $f_n:A\rightarrow C$ by $z\mapsto z^n$ for any $n\in \N$. Then $\mult_0 f_n^*(D) = n$, which is not bounded.
\end{example}

However, the appropriate conditions cannot be stated at this point as they are best phrased in terms of log differential forms. See Theorem \ref{thm:main} for details.

\section{Log Algebras and Schemes}

In order to develop log differential forms and jet spaces some background on log algebras and schemes is necessary. For a much more complete treatment of the subject matter see \cite{ogus06}. In some sense the main use for log algebras and log schemes in the context of this paper will be to keep track of units. Most of the examples in the previous section are statements about units, or sums of units, after removing certain divisors. Log geometry will allow us to keep track of would-be units without having to remove any points from the scheme. Before proceeding it is necessary to establish some conventions and definitions.

\begin{mydef}
A \textit{monoid} is a triple $(M, \star, e_M)$ consisting of a set $M$ along with an associative binary operation $\star: M\times M\rightarrow M$ and a two sided identity element $e_M\in M$. Throughout this paper all monoids will be assumed to be commutative.  When it is clear the set will be used to represent the monoid. Additionally $+$ or $\cdot$ will frequently be used to denote the binary operation.
\end{mydef}

\begin{example}
Any group is a monoid. However, the positive integers $(\Z_{>0}, \cdot, 1)$ are also a monoid under multiplication despite not being a group.
\end{example}

\begin{mydef}
A monoid $M$ is said to be \textit{integral} if $m+n = m'+n$ implies that $m=m'$. Some authors will also use the term \textit{cancellative}.
\end{mydef}

\begin{example}
The integers under multiplication $(\Z, \cdot, 1)$ are not integral as $2\cdot 0 = 3\cdot 0$ but $2\neq 3$. Most monoids used in this paper will be integral.
\end{example}

\begin{mydef}
A \textit{morphism of monoids} is a map $\phi:M\rightarrow N$ such that $\phi(e_M) = e_N$ and $\phi(a\star b) = \phi(a)\star\phi(b)$ for all $a,b\in M$.
\end{mydef}

\begin{example}
The natural logarithm $\ln: (\Z_{>0},\cdot,1)\rightarrow (\R,+, 0)$ is a morphism of monoids.
\end{example}

\begin{mydef}
Let $A$ be a ring and denote by $A^\times$ the multiplicative monoid
$(A, \cdot, 1)$ and denote by $A^*$ the group of units of
$A^\times$. Similarly for a monoid $M$, define $M^*$ to be the group of units of $M$.
\end{mydef}

\begin{example}
For the ring of integers $\Z^\times = \Z$ as a set, but $\Z^* = \{-1,1\}$.
\end{example}

\begin{mydef}
Let $A$ be a ring, then a \textit{pre-log structure} on $A$ is a
pair $(M_A,\alpha_A)$ where $M_A$ is a monoid and $\alpha_M$ is a
homomorphism of monoids $\alpha_A:M_A\rightarrow A^\times$. If in
addition $\alpha_A:\alpha_A^{-1}\left(A^*\right)\rightarrow A^*$ is
an isomorphism we call $(M_A,\alpha_A)$ a \textit{log structure} on
$A$. Every ring comes equipped with a \textit{trivial log structure} given by the inclusion $A^*\rightarrow A$.
\end{mydef}

\begin{mydef}
A \textit{pre-log algebra} is a triple $(A,M_A,\alpha_A)$ where $A$ is a
ring and $(M_A,\alpha_A)$ is a pre-log structure on $A$. If in addition $(M_A,\alpha_A)$ is a log structure we call $(A,M_A,\alpha_A)$ a \textit{log algebra}. When it is clear
$A$ will be used to represent the triple.
\end{mydef}

\begin{mydef}
A \textit{morphism of pre-log algebras} is a pair $(f,f^\flat):A\rightarrow B$
where $f:A\rightarrow B$ is a ring homomorphism and
$f^\flat:M_A\rightarrow M_B$ is a monoid homomorphism such that
$f\circ \alpha_A = \alpha_B \circ f^\flat$, i.e. the following diagram commutes
\[
\xymatrix{
M_A \ar[d]_{\alpha_A} \ar[r]^{f^\flat} & M_B \ar[d]^{\alpha_B}\\
A \ar[r]_{f} & B
}
\]
When it is clear $f$ will be used to represent the pair. If $A$ and $B$ are both log algebras $f$ will be called a \textit{morphism of log algebras}, and we will also say that $B$ is a log algebra over $A$.
\end{mydef}

\begin{mydef}
Let $Q_1,Q_2$ be monoids and define the \textit{direct sum of} $Q_1$ \textit{and} $Q_2$, denoted $Q_1\oplus Q_2$, to be the set $Q_1\times Q_2$ where the binary operation is carried out component-wise.
\end{mydef}

It is typically not convenient to have to list elements of a log structure. Description of the monoid itself can be cumbersome. The notion of an amalgamated sum will allow for an easier description of monoids by giving a method to construct log structures from pre-log structures.

\begin{mydef}
Let $P, Q_1, Q_2$ be monoids with morphisms $u_i: P\rightarrow Q_i$. Define the \textit{amalgamated sum}, denoted $Q_1\oplus_P Q_2$, to be the monoid equipped with morphisms $v_i: Q_i\rightarrow Q_1\oplus_P Q_2$ that makes the following diagram cocartesian.
\[
\xymatrix{
Q_1\oplus_P Q_2 & Q_2 \ar[l]_{v_2}\\
Q_1 \ar[u]^{v_1}  & P \ar[u]_{u_2} \ar[l]^{u_1}
}
\]
\end{mydef}

\begin{prop}
Let $P, Q_1, Q_2$ be monoids with morphisms $u_i:P\rightarrow Q_i$, and suppose that either $P$ or $Q_2$ is a group. Then $Q_1\oplus_P Q_2 \cong Q_1\oplus Q_2/\sim$ where the equivalence relation is given by $(q_1,q_2) \sim (q_1',q_2')$ if there exists $p,p'\in P$ such that
\[
(q_1+u_1(p), q_2-u_2(p)) = (q_1'+u_1(p'), q_2'-u_2(p')).
\]
The morphisms $v_i: Q_i\rightarrow Q_1\oplus Q_2/\sim$ are given by $v_1(q_1) = (q_1,0)$, $v_2(q_2) = (0,q_2)$ for all $q_i\in Q_i$.
\end{prop}

\begin{rem}
Throughout this paper equivalence classes will be referred to by their elements.
\end{rem}

\begin{proof}
First note that $-u_2(p)$ and $-u_2(p')$ are both well defined in either case. If $P$ is a group then the image of any $p\in P$ is a unit under $u_i$. On the other hand if $Q_2$ is a group then $u_2(p)$ is invertible for all $p\in P$. It must be verified that $\sim$ is in fact an equivalence relation.
\begin{enumerate}
\item[i.] Reflexive: Let $p=p'=0$ then $(q_1+u_1(0),q_2-u_2(0)) = (q_1+u_1(0), q_2+u_2(0))$ for all $q_i\in Q_i$. Therefore $\sim$ is reflexive.

\item[ii.] Symmetric: This is immediate by the symmetry of equality.

\item[iii.] Transitive: Suppose $(q_1,q_2)\sim (q_1',q_2')$ and $(q_1',q_2')\sim (q_1'', q_2'')$. Then there exists $p, p', m, m'\in P$ such that
    \[
    \begin{split}
    (q_1+u_1(p), q_2-u_2(p)) &= (q_1'+u_1(p'), q_2'-u_2(p'))\\
    (q_1'+u_1(m), q_2'-u_2(m)) &= (q_1''+u_1(m'), q_2''-u_1(m')).
    \end{split}
    \]
    Combining these expressions gives
    \[
    (q_1+u_1(p+m), q_2-u_2(p+m)) = (q_1''+u_1(p'+m'), q_2-u_2(p'+m')).
    \]
    Therefore $(q_1,q_2)\sim(q_1'',q_2'')$ and $\sim$ is transitive.
\end{enumerate}
Next we must verify that $v_1\circ u_1 = v_2\circ u_2$. Let $p\in P$, then in $Q_1\oplus Q_2$
\[
(u_1(p)+u_1(0), 0-u_2(0)) = (0+u_1(p), u_2(p)-u_2(p))
\]
or $(u_1(p), 0)\sim(0,u_2(p))$ and $v_1\circ u_1(p) = v_2\circ u_2(p)$ for all $p\in P$.

To complete the proof we must verify that $Q_1\oplus Q_2/\sim$ is in fact universal. Let $Q$ be a monoid and $w_i: Q_i\rightarrow P$ be such that $w_1\circ u_1 = w_2\circ u_2$. Define $h: (Q_1\oplus Q_2/\sim)\rightarrow Q$ by $(q_1,q_2) \mapsto w_1(q_1) + w_2(q_2)$. Then we claim $h$ is well defined, unique and that the following diagram commutes
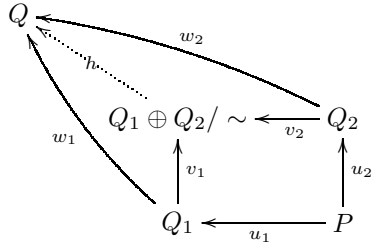
\[
\xymatrix{
Q\\
& Q_1\oplus Q_2/\sim \ar@{.>}[ul]|-h & Q_2 \ar[l]^<<<<{v_2} \ar@/_/[ull]_{w_2}\\
& Q_1 \ar[u]_{v_1} \ar@/^/[uul]^{w_1} & P \ar[l]^{u_1} \ar[u]_{u_2}
}
\]
Let $(q_1,q_2)\sim (q_1',q_2')$, then there exists $p, p'\in P$ such that
\[
(q_1+u_1(p),q_2-u_2(p)) = (q_1'+u_1(p'), q_2'-u_2(p'))
\]
and
\[
\begin{split}
h(q_1,q_2) &= w_1(q_1)+w_2(q_2)\\
&= w_1(q_1)+w_2(q_2)+w_1(u_1(p))-w_2(u_2(p))\\
&= w_1(q_1+u_1(p))+w_2(q_2-u_2(p))\\
&= w_1(q_1'+u_1(p')) + w_2(q_2'-u_2(p'))\\
&= w_1(q_1')+w_2(q_2')+w_1(u_1(p'))-w_2(u_2(p'))\\
&= h(q_1',q_2').
\end{split}
\]
Therefore $h$ is well defined and by construction the diagram commutes. For uniqueness let $h': Q_1\oplus Q_2/\sim\rightarrow Q$ be any other morphism making the diagram commute, then
\[
\begin{split}
h'(q_1,q_2) &= h'(v_1(q_1)+v_2(q_2))\\
&= h'(v_1(q_1)) + h'(v_2(q_2))\\
&= w_1(q_1) + w_2(q_2)
\end{split}
\]
and $h'=h$.
\end{proof}

\begin{rem}
The hypothesis that $Q_2$ is a group can be replaced by the hypothesis that $Q_1$ is a group by just switching the ordered pairs in the proof. In general it is difficult to describe the amalgamated sum, fortunately throughout this paper we will only need amalgamated sums where at least one of the monoids involved is a group. The construction is particularly simple in the event that $P$ is a group, in this case the equivalence relation can be shortened to $(q_1, q_2)\sim (q_1+u_1(p), q_2+u_2(-p))$ for all $q_i\in Q_i$ and $p\in P$.
\end{rem}

With the amalgamated sum in hand there are a variety of ways to define log structures. One method is to push forward a log structure via a morphism of rings, or in the category of schemes pull back the log structure.

\begin{example}\label{ex:strict}
Let $A$ be a log algebra, $B$ be a ring and $f:A\rightarrow B$ a morphism of rings. The log structure $M_A$ on $A$ gives rise to a log structure $\alpha_B:M_A\oplus_{A^*} B^* \rightarrow B$. The morphism, $\alpha_B$, comes from the universal property of amalgamated sums and is defined by $(m,b)\mapsto f(\alpha_A(m))b$.
\end{example}

This example is common enough that it is worthwhile to give it a definition.

\begin{mydef}
A morphism of log algebras $(f,f^\flat):A\rightarrow B$ is said to
be \textit{strict} if the induced map $M_A\oplus_{A^*}B^*\rightarrow M_B$ is
an isomorphism.
\end{mydef}

Another way log structures commonly arise is by associating a log structure to a pre-log structure.

\begin{prop}\label{existslog}
Let $\alpha: M\rightarrow A$ be a pre-log structure on a ring $A$. Then $\alpha_A: M\oplus_{\alpha^{-1}(A^*)} A^*\rightarrow A$ is a log structure on $A$, where $\alpha_A$ comes from the universal property of amalgamated sums. Moreover this log structure is universal in the sense that for any log structure $\alpha_A': M'\rightarrow A$ and commutative diagram of monoids
\[
\xymatrix{
M \ar[r]^\phi \ar[dr]_{\alpha} & M' \ar[d]^{\alpha_A'}\\
& A
}
\]
there exists a unique morphism $h:M\oplus_{\alpha^{-1}(A^*)} A^*\rightarrow M'$ such that $\phi = h\circ v_1$, where $v_1:M\rightarrow M\oplus_{\alpha^{-1}(A^*)} A^*$ is the morphism coming from the definition of amalgamated sum.
\end{prop}

\begin{proof}
It must be verified that $\alpha_A: M\oplus_{\alpha^{-1}(A^*)} A^*\rightarrow A$ is a log structure. Let $(m, a)\in M\oplus_{\alpha^{-1}(A^*)} A^*$ and suppose that $\alpha_A(m,a) = q \in A^*$. Since $q,a\in A^*$, $\alpha(m) = q/a\in A^*$ or $m\in \alpha^{-1}(A^*)$. In particular $(m,a)+(u_1(0),1/\alpha(0)) = (0,q) + (u_1(m), 1/\alpha(m))$, that is $(m,a)\sim (0,q)$ and $\alpha_A:\alpha_A^{-1}(A^*)\rightarrow A^*$ is injective. On the other hand for any $a\in A^*$, $\alpha_A(0,a) = a$ and thus $\alpha_A$ induces an isomorphism of units $\alpha_A: \alpha_{A}^{-1}(A^*)\rightarrow A$ making it a log structure.

Let $M'$ and $\phi$ be as in the statement of the theorem and consider the diagram
\[
\xymatrix{
M' & A^* \ar[l]_{(\alpha'_A)^{-1}}\\
M \ar[u]^\phi & \alpha^{-1}(A^*) \ar[l]^{u_1} \ar[u]_\alpha
}
\]
Let $m\in \alpha^{-1}(A^*)$, by hypothesis $\alpha'_A\circ\phi\circ u_1(m) =\alpha\circ u_1(m)$. However $\alpha\circ u_1(m)\in A^*$ by assumption and $\alpha'_A$ induces an isomorphism on the group of units, therefore
\[
\phi\circ u_1(m) = (\alpha'_A)^{-1}\circ\alpha\circ u_1(m) = (\alpha'_A)^{-1}\circ\alpha(m)
\]
and the diagram commutes. By the universal property of amalgamated sums the existence and uniqueness of $h:M\oplus_{\alpha^{-1}(A^*)} A^*\rightarrow M'$ is guaranteed.
\end{proof}

With this construction in hand we are able to define log structures via pre-log structure. In truth this is just a generalization of Example \ref{ex:strict}. In that example $f(\alpha_A)$ defined the pre-log structure on $B$. We can also define a log structure by simply giving a set of elements of the ring.

\begin{example}
Let $A = k[x,y]$. Then we can talk about the log structure associated to the elements $x$ and $y-x^2$. The elements $x$ and $y-x^2$ generate a sub-monoid of $A^\times$ and therefore a pre-log structure. By the previous proposition there is a universal log structure which can be associated to the pre-log structure. In this case we get $\alpha_A:\N^2\oplus k^*\rightarrow A$, $((m,n),c)\mapsto cx^m(y-x^2)^n$.
\end{example}

Localized log algebras can also be defined via the amalgamated sum.

\begin{mydef}\label{def:monoidlocal}
Let $(A,M_A,\alpha_A)$ be a log algebra, $S\subset A$ to be a multiplicative subset and $\varphi_S:A\rightarrow S^{-1}A$ the natural map in the category of rings. Define the \textit{localized log algebra of} $A$ \textit{by} $S$ to be
\[
(S^{-1}A, S^{-1}M_A, S^{-1}\alpha_A)
\]
where
\[
S^{-1}M_A := M_A\oplus_{(\varphi_S\circ\alpha_A)^{-1}((S^{-1}A)^*)}(S^{-1}A)^*
\]
and $S^{-1}\alpha_A(m,a) := \varphi_S(\alpha_A(m))a$. Additionally define $\varphi_S^\flat: M_A\rightarrow S^{-1}M_A$ to be the unique morphism coming from the universal property of amalgamated sums.

Following the standard notational conventions for rings and modules, define
\[
(M_A)_p := \{a\in A: a\notin p\}^{-1}M_A
\]
for any prime ideal $p\subset A$ and
\[
(M_A)_f := \{f^m\in A: m\in\N\}^{-1}M_A
\]
for any $f\in A$.
\end{mydef}

\begin{rem}
The above construction defines a log algebra by Proposition \ref{existslog}.
\end{rem}

\begin{prop}\label{quotientuniv}
Let $A$ be a log algebra and $S\subset A$ a multiplicative subset. Then for any morphism of log algebras $f:A\rightarrow R$ such that $f(S)\subset R^*$, $f$ factors uniquely through $\varphi_S:A\rightarrow S^{-1}A$ in the category of log algebras.
\end{prop}

\begin{proof}
By the universal property of localized rings in the category of rings there exists a unique morphism $h:S^{-1}A\rightarrow R$ such that $f = h\circ \varphi_S$. For the case of the log structure consider the diagram
\[
\xymatrix{
M_R & (S^{-1}A)^*\ar[l]_{\alpha_R^{-1}\circ h}\\
M_A \ar[u]^{f^\flat}& (\varphi_S\circ\alpha_A)^{-1}((S^{-1}A)^*) \ar[l] \ar[u]_{\varphi_S\circ\alpha_A}
}
\]
We have that $\alpha_R^{-1}\circ h\circ \varphi_S\circ\alpha_A = \alpha_R^{-1}\circ f\circ \alpha_A$. However since $f$ is a morphism of log algebras $\alpha_R^{-1}\circ f\circ \alpha_A = f^\flat$ when well defined. However for any $m\in (\varphi_S\circ\alpha_A)^{-1}((S^{-1}A)^*)\subset M_A$,
$f\circ\alpha_A(m)\in R^*$ and the morphism is well defined. Therefore the diagram commutes and by the universal property of amalgamated sums $\varphi_S^\flat:S^{-1}M_A\rightarrow M_R$ is uniquely determined.
\end{proof}

While we have all of the definitions needed for constructing log Hasse-Schmidt rings, one more definition will be useful as these are among the most common class of log algebras.

\begin{mydef}
A log structure $\alpha_M: M_A\rightarrow A$ is said to be \textit{finitely generated over} $A^*$ if there exists a commutative diagram
\[
\xymatrix{
P \ar[r]^\phi \ar[dr]_{\alpha_P} & M_A \ar[d]^{\alpha_M}\\
& A
}
\]
where $P$ is a finitely generated monoid and $\phi$ induces an isomorphism $P\oplus_{\alpha_P^{-1}(A^*)}A^*\cong M_A$. If in addition $P$ is integral, we say that $A$ is \textit{fine}.
\end{mydef}

Now that the theory of log algebras have been defined sufficiently for the author's applications, log schemes can be appropriately defined.

\begin{mydef}
A \textit{pre-log structure} on a scheme $X$ is a pair $(M_X,\alpha_X)$ where $M_X$ is a sheaf of monoids and $\alpha_X:M_X\rightarrow \Oc_X$ is a morphism of sheaves of monoids. A pre-log structure is called a \textit{log structure} if in addition $\alpha_X:\alpha_X^{-1}(\Oc_X^*)\rightarrow \Oc_X^*$ is an isomorphism of sheaves.
\end{mydef}

\begin{mydef}
A \textit{morphism of pre-log structures} $f:M_X\rightarrow M_X'$ on a scheme $X$ is a morphism of sheaves of monoids on $X$ such that the following diagram commutes.
\[
\xymatrix{
M_X \ar[d]_f\ar[dr]^{\alpha_X} &\\
M_{X}'\ar[r]_{\alpha_X'} & \Oc_X
}
\]
\end{mydef}

\begin{mydef}
A \textit{log scheme} is a triple $(X, M_X, \alpha_X)$ where $X$ is a scheme and $(M_X,\alpha_X)$ is a log structure on $X$. When it is clear $X$ will be used to represent the triple.
\end{mydef}

\begin{mydef}
A \textit{morphism of log schemes} is a pair $(f,f^\flat): (X,M_X,\alpha_X)\rightarrow (Y,M_Y,\alpha_Y)$ where $f:X\rightarrow Y$ is a morphism of schemes, $f^\flat:M_Y\rightarrow f_*(M_X)$ is a morphism of pre-log structures on $Y$, and the following diagram commutes.
\[
\xymatrix{
M_Y\ar[r]^{f^\flat}\ar[d]_{\alpha_Y} & f_*(M_X)\ar[d]^{f_*(\alpha_X)}\\
\Oc_Y\ar[r]^{f^\#} & f_*(\Oc_X)
}
\]
When it is clear $f$ will be used to represent the pair.
\end{mydef}

Just as quasi-coherent sheaves of modules play an important role in the theory of schemes, quasi-coherent log structures will be our primary object of study with regard to log schemes. Using a similar construction to the sheafification of a module over an affine scheme we can define a log structure on an affine scheme associated to a pre-log structure on the ring of regular functions. Despite the similarities in definition there are subtle differences in the theory as demonstrated in Example \ref{badmonoidsheaf}.

\begin{mydef}\label{sheafdef}
Let $\alpha:M\rightarrow A$ be a pre-log structure on a ring $A$. For $p\in\spec A$, we write $\alpha^{-1}(A_p^*)$ for the inverse image of $A_p^*$ under the composition $M\rightarrow A\rightarrow A_p$. For each open set $U\subset X :=\spec A$ define $\widetilde{M}(U)$ to be the set of functions
\[
s:U\rightarrow \coprod_{p\in U} M\oplus_{\alpha^{-1}(A_p^*)} A_p^*
\]
such that
\begin{enumerate}
\item[i.] for every $p\in U$, $s(p)\in M\oplus_{\alpha^{-1}(A_p^*)} A_p^*$ and

\item[ii.] for every $p\in U$ there exists an open affine neighborhood $V\subset U$ containing $p$ and $t\in M\oplus_{\alpha^{-1}(\Oc_X(V)^*)} \Oc_X(V)^*$ such that for all $q\in V$ the image of $t$ in the natural map $M\oplus_{\alpha^{-1}(\Oc_X(V)^*)} \Oc_X(V)^*\rightarrow M\oplus_{\alpha^{-1}(A_q^*)} A_q^*$ is equal to $s(q)$.
\end{enumerate}
Additionally define $\widetilde{\alpha}: \widetilde{M}\rightarrow \Oc_X$ by
$\widetilde{\alpha}(s)(p) := \alpha_{A_p}\circ s(p)$.
\end{mydef}

\begin{prop}
Let everything be as in Definition \ref{sheafdef}. Then $\widetilde{M}$ defines a log structure on $X$.
\end{prop}

\begin{proof}
By construction $\widetilde{M}$ is a sheaf of monoids, so it suffices to check that $\widetilde{\alpha}$ is a well defined morphism of sheaves of monoids inducing an isomorphism on units. Let $U\subseteq X$ be a non-empty open subset, $s\in \widetilde{M}(U)$ a section, and $p\in U$ a point. Let $t=(m,v)$ and $V$ be as in property ii of the previous definition where $m\in M$ and $v$ is a unit on $V$. Then for every $q\in V$ we have
\[
\begin{split}
\widetilde{\alpha}(s)(q) &= \alpha_{A_q}(t)\\
&= \alpha_{A_q}(m,v)\\
&= \alpha(m)v\\
&= \alpha_{\Oc_X(V)}(m,v)\\
&= \alpha_{\Oc_X(V)}(t)
\end{split}
\]
where the natural morphisms $\Oc_X(V)\rightarrow A_q$ have been omitted. Since $p\in U$ was arbitrary $\widetilde{\alpha}(s)$ defines a section of the structure sheaf on $U$ and therefore $\widetilde{M}$ is a pre-log structure on $X$.

To see that $\widetilde{M}$ is in fact a log structure on $X$ let $r\in \Oc_X(U)^*$, then
\[
r:U\rightarrow \coprod_{p\in U} A_p^*
\]
lifts uniquely to a function
\[
s: U\rightarrow \coprod_{p\in U} A_p^*\rightarrow \coprod_{p\in U} M\oplus_{\alpha^{-1}(A_p^*)} A_p^*
\]
because $\alpha_{A_p}:M\oplus_{\alpha^{-1}(A_p^*)} A_p^*\rightarrow A_p$ is a log structure for every $p\in U$ by Proposition \ref{existslog}. The function $s$ trivially satisfies property i in Definition \ref{sheafdef} and for property ii let $V=U$ and $t=(0,r)$. Therefore $s\in\widetilde{M}(U)$, since $U$ and $r$ were arbitrary $\widetilde{M}$ is a log structure on $X$.
\end{proof}

\begin{mydef}
A log scheme $X$ is said to be \textit{quasi-coherent} if there exists an open affine cover $\{\spec A_i\}_{i\in I}$ and pre-log structures $\{\alpha_i: M_i\rightarrow A_i\}$ such that $M_X|_{\spec A_i}\cong \widetilde{M_i}$ are isomorphic as sheaves of monoids for every $i\in I$. If the $M_i$ are all finitely generated (respectively integral, respectively both) $X$ is said to be \textit{coherent} (respectively \textit{integral}, respectively \textit{fine}).
\end{mydef}

In the case of modules over a ring the sheafification process is fully determined by localizing the module. This is precisely what makes quasi-coherent sheaves of modules useful. In particular if $N$ is a module over a ring $A$, $\widetilde{N}$ the corresponding sheaf on modules on $\spec{A}$ and $f\in A$ then $\widetilde{N}(D(f)) = N\otimes_A A_f$. Unfortunately the same does not hold in the case of log structures, as the following example demonstrates.

\begin{example}\label{badmonoidsheaf}
Let $k$ be a field of characteristic not equal to $2$, $X = \spec k[x]$, $M = \N\oplus\N$ and $\alpha:M\rightarrow X$ defined by $\alpha(p,q) = (x-2)^p(x+1)^p(x+2)^q(x-1)^q$. The variety $X$ is covered by the open sets $U_1 = D(x^2-1)$ and $U_2 = D(x^2-4)$. Define $s_1=(1,0,(x-1)/(x+1))\in M\oplus_{\alpha^{-1}(\Oc_X(U_1)^*)} \Oc_X(U_1)^*$ and $s_2=(0,1,(x-2)/(x+2))\in M\oplus_{\alpha^{-1}(\Oc_X(U_2)^*)} \Oc_X(U_2)^*$. On the intersection $U_1\cap U_2$ we have $M\oplus_{\alpha^{-1}(\Oc_X(U_1\cap U_2)^*)} \Oc_X(U_1\cap U_2)^* \cong (k[x]_{(x^2-1)(x^2-4)})^*$, and $s_1=s_2=(x-1)(x-2)$. Therefore $s_1$ and $s_2$ glue to form a section $s\in \widetilde{M}(X)$. However, every element of $M\oplus_{\alpha^{-1}(k[x]^*)}k[x]^*$ when restricted to $U_1\cap U_2$ is of the form $c(x-2)^p(x+1)^p(x+2)^q(x-1)^q$ where $c\in k^*$ and $p,q\in \N$. Therefore $s\notin M\oplus_{\alpha^{-1}(k[x]^*)}k[x]^*$, and in particular $\widetilde{M}(X) \neq M\oplus_{\alpha^{-1}(k[x]^*)}k[x]^*$.
\end{example}

That is to say that the sheafification process adds additional sections. One cannot obtain a sheaf by simply constructing the minimal log structure on each open subset of $X$. Fortunately not all is lost.

\begin{prop}\label{monoidstalk}
Let $\alpha:M\rightarrow A$ be a pre-log structure on a ring $A$. Then for every $p\in \spec A$,
\[
\widetilde{M}_p \cong M\oplus_{\alpha^{-1}(A_p^*)}A_p^*.
\]
\end{prop}

\begin{proof}
Let $(m,u)\in M\oplus_{\alpha^{-1}(A_p^*)}A_p^*$, then there exists $a,b\in A$ such that $u = a/b$ and $a,b\notin p$. Therefore $u=a/b$ is a unit on the open set $D(ab)$ and $(m,u)\in M\oplus_{\alpha^{-1}(A_{ab}^*)}A_{ab}^*$. Since $(m,u)$ was arbitrary this construction defines a morphism
\[
M\oplus_{\alpha^{-1}(A_p^*)}A_p^*\rightarrow \widetilde{M}_p.
\]
On the other hand let $s\in\widetilde{M}_p$, then by definition $s\in \widetilde{M}(D(g))$ for some $g\notin p$. By the universal property of amalgamated sums there exists a unique morphism
\[
M\oplus_{\alpha^{-1}(A_g^*)}A_g^*\rightarrow M\oplus_{\alpha^{-1}(A_p^*)}A_p^*.
\]
Therefore $s$ determines an element of $M\oplus_{\alpha^{-1}(A_p^*)}A_p^*$.  Since $s$ was arbitrary this construction defines a morphism $\widetilde M_p\rightarrow M\oplus_{\alpha^{-1}(A_p^*)}A_p^*$,
which is inverse to the previously constructed morphism.
\end{proof}

\section{Log Derivations}

Before defining higher log derivations, it is worthwhile to review the definition of log derivations.

\begin{mydef}
Let $B$ be a log algebra over $A$ with morphism $g_B:A\rightarrow B$,
and $E$ a $B$-module. Then a \textit{log derivation} from $B$ to $E$ over $A$ is a pair $(D,\delta)$ satisfying the following
\begin{enumerate}
\item[i.] $D:B\rightarrow E$ is a morphism of $A$-modules.

\item[ii.] $D(xy) = xD(y) + yD(x)$ for all $x,y\in B$.

\item[iii.] $\delta: M_B\rightarrow E$ is a morphism of monoids satisfying $D(\alpha_B(m)) = \alpha_B(m) \delta(m)$ for all $m\in M_B$.
\item[iv.] $\delta(g_B^\flat(a)) =0$ for all $a\in A$.
\end{enumerate}
Denote by $\Der_{B/A}(E)$ the set of all log derivations from $B$ to $E$ over $A$.
\end{mydef}

There will not be many times that we need to directly deal with log derivations. However they will naturally appear as a part of higher log derivations. As such the proof of the following theorem is omitted.

\begin{thm}\label{existsuniv}
Let $B$ be a log algebra over $A$ with morphism $g_B:A\rightarrow B$. Then the functor of $B$-modules $E\mapsto \Der_{B/A}(E)$ is representable by a universal object $\Omega_{B/A}(M_B/M_A)$ with morphisms $d:B\rightarrow \Omega_{B/A}(M_B/M_A)$ and $\partial: M_B\rightarrow \Omega_{B/A}(M_B/M_A)$.
\end{thm}

\begin{proof}
See \cite{ogus06} Proposition IV.1.1.6.
\end{proof}

There are of course many ways such a definition can be generalized. The goal here is to develop the notion of log jet spaces as used by Noguchi and Winkelmann in \cite{nogu1} in the style and rigor of Vojta in \cite{vojt1}. Therefore we make the following definition.

\begin{mydef}\label{def:logder}
Let $B$ and $R$ be log algebras over $A$ with morphisms
$g_B:A\rightarrow B$ and $g_R:A\rightarrow R$, then a \textit{higher
log derivation of order n} from $B$ to $R$ over $A$ is a
$(2n+2)$-tuple $(D_0, \delta_0, \ldots, D_n, \delta_n)$ such that
\begin{enumerate}
\item[i.] $D_0: B\rightarrow R$ is a homomorphism of log algebras over $A$.

\item[ii.] $D_k: B\rightarrow R$ are $A$-module homomorphisms satisfying the divided Leibniz Identity
\[
D_k(x\cdot y) = \sum_{i+j=k} D_i(x)D_j(y)
\]
for all $x,y \in B$ and all $k = 0,\ldots, n$.

\item[iii.] $\delta_k:M_B\rightarrow R$ are maps satisfying
\[
D_k(\alpha_B(m)) = D_0(\alpha_B(m))\delta_k(m)
\]
for all $m\in M_B$ and all $k = 0, \ldots, n$.

\item[iv.] $\delta_0 = 1$ as a function and
\[
\delta_k(m+p) = \sum_{i+j=k} \delta_i(m)\delta_j(p)
\]
for all $m, p\in M_B$ and all $k = 0, \ldots, n$.

\item[v.] $\delta_k(g_B^\flat(m)) = D_k(g_B(a)) = 0$ for all $m\in
M_A, a\in A$ and all $k=1,\ldots, n$.
\end{enumerate}
Denote by $\Der^n_A(B, R)$ the set of all higher log derivations of order $n$ over $A$.
\end{mydef}

\begin{rem}
Condition iv above is not necessary if $\alpha_R(M_R)$ contains no
zero divisors since $D_0(\alpha_B(m+p))\in \alpha_R(M_R)$ for all
$m,p \in M_B$ and
\[
\begin{split}
D_0(\alpha_B(m+p))\delta_k(m+p) &= D_k(\alpha_B(m+p))\\
&= D_k(\alpha_B(m)\alpha_B(p))\\
&= \sum_{i+j=k}D_i(\alpha_B(m))D_j(\alpha_B(p))\\
&=
\sum_{i+j=k}D_0(\alpha_B(m))\delta_i(m)D_0(\alpha_B(p))\delta_j(p)\\
&= D_0(\alpha_B(m+p))\sum_{i+j=k}\delta_i(m)\delta_j(p)
\end{split}
\]
by properties i, ii and iii. Additionally $D_k(g_B(a)) = 0$ in v is
not necessary as it follows from ii.
\end{rem}

\begin{rem}\label{rem:oddlog}Condition iii may seem unnatural for $k>1$, but later it will be demonstrated that the differentials one would expect via application of the product rule to first order log differentials can be recovered over fields of characteristic $0$. See Theorem \ref{thm:gendiff} and Example \ref{ex:gendifff} for details. As seen in \cite{vojt1} the appropriate definition for higher log derivations in constructing jet spaces uses divided differentials. If the chain rule from calculus were carried over and applied directly to divided differentials, the relation would be
\[
\alpha_B(2m)\delta_2(m) = \alpha_B(m)D_2(\alpha_B(m))-\frac{1}{2}D_1(\alpha_B(m))^2,
\]
which is not well defined in characteristic $2$.
\end{rem}

\begin{mydef}
Let $R$ be a log algebra over $A$ and define the log structure on
the ring $R[t]/t^{n+1}$ to be
\[
\widehat{M}^n_R = M_R\times R^n
\]
as a set. Define a monoid law on $\widehat M_R^n$ by
\[
(m,r_1,\ldots,r_n)+(p, q_1,\ldots, q_n) = \left(m+p, \sum_{i+j=1}
r_i q_j,\ldots, \sum_{i+j=n} r_iq_j\right),
\]
where $r_0 = q_0 = 1$ and the sums and products are taken in their
respective monoid or ring. It can easily be verified that
$\widehat{M}^n_R$ is in fact a monoid. Define
$\widehat{\alpha}^n_R:\widehat{M}^n_R\rightarrow R[t]/t^{n+1}$ by
\[
(m,r_1,\ldots,r_n) \mapsto \alpha_R(m)(1+r_1t+\cdots+r_nt^n),
\]
where $R[t]/t^{n+1}$ is viewed as an $R$-module. Finally
define $g_{R[t]/t^{n+1}}^\flat:M_A\rightarrow \widehat{M}^n_R$ by
\[
m \mapsto (g_R^\flat(m),0,\ldots,0),
\]
which along with the composition $g_{R[t]/t^{n+1}}:A\rightarrow
R\rightarrow R[t]/t^{n+1}$ turns $R[t]/t^{n+1}$ into a log algebra
over $A$.
\end{mydef}

\begin{rem}
$\widehat{M}^n_R$ is in fact just $M_R\oplus_{R^*} \left(R[t]/t^{n+1}\right)^*$, but
it will be useful at times to be able to explicitly write down the monoid law.
\end{rem}

\begin{thm}\label{thm:rep}
Let $B$ and $R$ be log algebras over $A$. For each
$(D_0,\delta_0,\ldots, D_n,\delta_n)\in \Der^n_A(B,R)$
define $\phi: B\rightarrow R[t]/t^{n+1}$ by
\[
b\mapsto D_0(b)+D_1(b)t+\cdots+D_n(b)t^n
\]
and $\phi^\flat: M_B\rightarrow \widehat{M}_R^n$ by
\[
m\mapsto (D_0^\flat(m), \delta_1(m),\ldots, \delta_n(m)).
\]
Then $\phi$ is well defined and lies in
$\Hom_A(B,R[t]/t^{n+1})$ in the category of log algebras
over $A$. Moreover the resulting map
\[
\Der^n_A(B,R)\rightarrow\Hom_A(B,R[t]/t^{n+1})
\]
is a bijection.
\end{thm}

\begin{proof}
Let $x,y\in B$ and $\phi$ be as above. Then
\[
\begin{split}
\phi(xy) &= D_0(xy)+D_1(xy)t+\cdots+D_n(xy)t^n\\
&=
D_0(x)D_0(y)+\sum_{i+j=1}D_i(x)D_j(y)t+\cdots+\sum_{i+j=n}D_i(x)D_j(y)t^n\\
&=
\left(D_0(x)+D_1(x)t+\cdots+D_n(x)t^n\right)\left(D_0(y)+D_1(y)t+\cdots+D_n(y)t^n\right)\\
&= \phi(x)\phi(y)
\end{split}
\]
and for $a\in A$
\[
\begin{split}
\phi(g_B(a)) &= D_0(g_B(a))+D_1(g_B(a))t+\cdots+D_n(g_B(a))t^n\\
&= D_0(g_B(a))\\
&= g_{R[t]/t^{n+1}}(a).
\end{split}
\]
Since the $D_i$ are $A$-linear $\phi$ is also $A$-linear and by the
above it is in $\Hom_A(B,R[t]/t^{n+1})$ taken in the
category of $A$ algebras.

Similarly let $m,p\in M_B$ and $\phi^\flat$ as above. Then
\[
\begin{split}
\phi^\flat(m+p) &= \left(D_0^\flat(m+p),
\delta_1(m+p),\ldots,\delta_n(m+p)\right)\\
&= \left(D_0^\flat(m)+D_0^\flat(p),
\sum_{i+j=1}\delta_i(m)\delta_j(p), \ldots,
\sum_{i+j=n}\delta_i(m)\delta_j(p)\right)\\
&=
\left(D_0^\flat(m),\delta_1(m),\ldots,\delta_n(m)\right)+\left(D_0^\flat(p),\delta_1(p),\ldots,\delta_n(p)\right)\\
&= \phi^\flat(m)+\phi^\flat(p)
\end{split}
\]
and for $m\in M_A$
\[
\begin{split}
\phi^\flat(g_B^\flat(m)) &= (D_0^\flat(g_B^\flat(m)), 0,\ldots,0)\\
&= (g_R^\flat(m),0,\ldots,0)\\
&= g_{R[t]/t^{n+1}}^\flat(m),
\end{split}
\]
so $\phi^\flat$ is a morphism of $M_A$-monoids.

Let $m\in M_B$. Then
\[
\begin{split}
\widehat{\alpha}_R^n(\phi^\flat(m)) &=
\widehat{\alpha}_R^n(D_0^\flat(m), \delta_1(m),\ldots,
\delta_n(m))\\
&=
\alpha_R(D_0^\flat(m))(1+\delta_1(m)t+\cdots+\delta_n(m)t^n)\\
&=
D_0(\alpha_B(m))(1+\delta_1(m)t+\cdots+\delta_n(m)t^n)\\
&= D_0(\alpha_B(m))+D_0(\alpha_B(m))\delta_1(m)t+\cdots+D_0(\alpha_B(m))\delta_n(m)t^n)\\
&= D_0(\alpha_B(m))+D_1(\alpha_B(m))t+\cdots+D_n(\alpha_B(m))t^n\\
&= \phi(\alpha_B(m))
\end{split}
\]
so $(\phi,\phi^\flat)$ is a log morphism over $A$.

Since the definition of higher log derivations requires $\delta_0$ to be identically $1$, the map $\Der_A^n(B,R)\rightarrow \Hom_A(B,R[t]/t^{n+1})$ is injective. This construction is reversible and an argument almost identical to the above proves that the $(D_0,\delta_0,\ldots, D_n,\delta_n)$ arising from a pair $(\phi,\phi^\flat)$ is a higher log derivation of order $n$. Therefore the map $\Der_A^n(B,R)\rightarrow \Hom_A(B,R[t]/t^{n+1})$ is a bijection.
\end{proof}

\begin{mydef} \label{def:hsdef}
Let $g_B:A\rightarrow B$ be a morphism of log algebras, and define the \textit{log Hasse-Schmidt ring of order n}, denoted $\HS^n_{B/A}(M_B/M_A)$, to be the log algebra over $B$
\[
B[b^{(i)}, m^{(i)}]_{b\in B; m\in M_B; i=1,\ldots,n}/I^n_{B/A},
\]
where $I^n_{B/A}$ is the ideal generated by the union of the sets
\begin{equation}
\{g_B(a)^{(i)} : a\in A; i=1,\ldots,n\}\textrm{,}
\end{equation}
\begin{equation}
\left\{g_B^\flat(m)^{(i)}: m\in M_A; i=1,\ldots,n\right\}\textrm{,}
\end{equation}
\begin{equation}\label{eq:sum}
\left\{(x+y)^{(i)}-x^{(i)}-y^{(i)}: x,y \in B;
i=1,\ldots,n\right\}\textrm{,}
\end{equation}
\begin{equation} \label{eq:product}
\left\{(xy)^{(i)}-\sum_{j+k=i}x^{(j)}y^{(k)}: x,y\in B;
i=1,\ldots,n\right\}\textrm{,}
\end{equation}
\begin{equation}\label{eq:monoidprod}
\left\{(m+p)^{(i)}-\sum_{j+k=i}m^{(j)}p^{(k)}: m,p\in M_b;
i=1,\ldots,n\right\}\textrm{,}
\end{equation}
\begin{equation}
\left\{\alpha(m)^{(i)}-\alpha(m)^{(0)}m^{(i)}:
m\in M_b, i=1,\ldots,n\right\}
\end{equation}
and $b^{(0)}$ is defined to be $b$ for all $b\in B$ and $m^{(0)}$ is
defined to be $1$ for all $m\in M_B$.
Define the  log structure $M^n_{B/A}$ to be the amalgamated sum
$M_B\oplus_{B^*}\HS^{n*}_{B/A}(M_B/M_A)$ along with the morphism
$\alpha^n_{B/A}:M^n_{B/A} \rightarrow \HS^n_{B/A}(M_B/M_A)$ given by
$(m,x)\mapsto \alpha_B(m)x$ where the obvious inclusion maps have
been omitted from notation.

Finally define the \textit{universal derivation}
$(d_0,\partial_0,\ldots,d_n,\partial_n)$ from $B$ to
$\HS^n_{B/A}(M_B/M_A)$ given by the maps $d_i(b) = b^{(i)} \pmod{I}$
and $\partial_i(m) = m^{(i)}\pmod{I}$, and $d_0^\flat(m) = (m,1)$
for all $b\in B, m\in M_B$ and  $i=0,\ldots n$.
\end{mydef}

\begin{rem}
The ring $B[b^{(i)}, m^{(i)}]_{b\in B; m\in M_B; i=1,\ldots,n}$ carries a natural grading where $b^{(i)}$ and $m^{(i)}$ each have weight $i$. The ideal $I^n_{B/A}$ is homogeneous, so $\HS^n_{B/A}(M_B/M_A)$ also has a natural grading.
\end{rem}

\begin{mydef}
For notational reasons define $\HS^n_{B/A} := \HS^n_{B/A}(B^*/A^*)$. This agrees with the definition of $\HS^n_{B/A}$ given in \cite{vojt1}.
\end{mydef}

\begin{prop}
Let $f:B\rightarrow C$ be a morphism of log algebras over $A$. Then the morphism
\[
B[b^{(i)}, m^{(i)}]_{b\in B; m\in M_B; i=1,\ldots,n}\rightarrow C[c^{(i)}, p^{(i)}]_{c\in C; p\in M_C; i=1,\ldots, n}
\]
given by $b^{(i)}\mapsto f(b)^{(i)}$, $m^{(i)}\mapsto f^{\flat}(m)^{(i)}$ induces a morphism of log algebras over $B$, $\HS^n_{B/A}(M_B/M_A)\rightarrow \HS^n_{C/A}(M_C/M_A)$, which by abuse of notation will be denoted $f$ as well.
\end{prop}

\begin{proof}
The fact that $f(I^n_{B/A})\subset I^n_{C/A}$ follows immediately from the fact that $f$ is a homomorphism of log algebras over $A$. Therefore $f:\HS^n_{B/A}\rightarrow \HS^n_{C/A}$ is well defined as a map of $B$-algebras. The morphism on log structures is uniquely determined by the universal property of amalgamated sums.
\end{proof}

\begin{thm}[First Universal Property]\label{thm:func}
Let $B$ and $R$ be log algebras over $A$ and $n\in \mathbb{N}$. For
each higher log derivation of order $n$
$(D_0,\delta_0,\ldots,D_n,\delta_n)$ from $B$ to $R$ there exists a
unique homomorphism $\phi:
\HS^n_{B/A}(M_B/M_A)\rightarrow R$ of log algebras over $A$ such that
\[
(D_0,\delta_0,\ldots,D_n,\delta_n) = (\phi\circ
d_0,\phi\circ\partial_0,\ldots,\phi\circ d_n,\phi\circ\partial_n).
\]
Moreover every higher log derivation of order $n$ arises in this way.
\end{thm}

\begin{proof}
Define $\phi:\HS^n_{B/A}(M_B/M_A)\rightarrow R$ by $\phi(b^{(i)}) =
D_i(b)$ for all $b\in B$ and $\phi(m^{(i)}) = \delta_i(m)$. The
morphism of monoids $\phi^\flat$, is defined to be the unique
morphism by the universal property of amalgamated sums, i.e.,
$\phi^\flat(m,p) = D_0^\flat(m)+\alpha_R^{-1}(\phi(p))$ which is
well defined by the definition of a log structure and the fact that
$p\in \HS^{n}_{B/A}(M_B/M_A)^*$ is a unit and thus $\phi(p)\in R^*$ is
as well. The properties of derivations imply that the kernel of
$\phi$ contains $I^n_{B/A}$. This map is unique since the value of
$\phi(b^{(i)}), \phi(m^{(i)})$ is uniquely determined by $D_i$ and
$\delta_i$.

On the other hand any morphism $\phi:\HS_{B/A}^n(M_B/M_A)\rightarrow R$ can be
pulled back by this same construction to a higher log derivation of order
$n$ from $B$ to $R$. Therefore we have a bijection
\[
\Der_A^n(B,R)\rightarrow \Hom_A(\HS_{B/A}^n(M_B/M_A), R).
\]
\end{proof}

\begin{cor}[Second Universal Property]\label{cor:functor}
Let $B$ and $R$ be log algebras over $A$ and let $n\in\mathbb{N}$. Then there is a natural bijection
\[
\Hom_A(\HS^n_{B/A}(M_B/M_A), R) \rightarrow
\Hom_A(B,R[t]/t^{n+1})
\]
which sends a morphism
$(\phi,\phi^\flat)\in\Hom_A(\HS^n_{B/A}(M_B/M_A), R)$ to the
morphism $B\rightarrow R[t]/t^{n+1}$ defined by
\begin{equation}\label{funceq}
\begin{split}
&x\mapsto \phi(d_0x)+\phi(d_1x)t+\cdots+\phi(d_nx)t^n\\
&m\mapsto (\phi^\flat(m),\phi(\partial_1 m),\ldots,\phi(\partial_n
m)).
\end{split}
\end{equation}
\end{cor}

\begin{proof}
This follows directly from \ref{thm:func} and \ref{thm:rep}.
\end{proof}

It is now possible to connect the definition of log derivations to higher log derivations of order $1$.

\begin{prop}
Let $B$ be a log algebra over $A$. Then exists a natural isomorphism of log algebras
\[
\HS^1_{B/A}(M_B/M_A) \cong \Sym\Omega_{B/A}(M_B/M_A)
\]
where the log structure $M_\Omega$ on $\Sym\Omega_{B/A}(M_B/M_A)$ is given by
\[
M_\Omega := M_B\oplus_{B^*}\left(\Sym\Omega_{B/A}(M_B/M_A)\right)^*.
\]
\end{prop}

\begin{proof}
It suffices to construct a higher log derivation of order $1$ from $B$ to the symmetric algebra $\Sym \Omega_{B/A}(M_B/M_A)$, and to verify that it satisfies the universal property of Theorem \ref{thm:func}. Define a higher log derivation of order $1$ from $B$ to $\Sym \Omega_{B/A}(M_B/M_A)$ by $(j, d, \partial)$, where $j:B\rightarrow \Sym^0 \Omega_{B/A}(M_B/M_A) = B$ is the identity map and $(d,\partial)$ are the universal log derivations coming from Theorem \ref{existsuniv} followed by the isomorphism $\Omega_{B/A}(M_B/M_A)\rightarrow \Sym^1 \Omega_{B/A}(M_B/M_A)$.

We must show that any higher log derivation of order $1$, $(D_0,D_1,\delta_1): B\rightarrow R$, factors uniquely through $\Sym \Omega_{B/A}(M_B/M_A)$. By Theorem \ref{existsuniv} there exists a unique morphism of $B$-modules $h:\Omega_{B/A}(M_B/M_A)\rightarrow R$ making the following diagram commute
\[
\xymatrix{
\Omega_{B/A}(M_B/M_A) \ar@{.>}[dr]|-h &\\
B \ar[u]^{(d,\partial)} \ar[r]_{(D_1,\delta_1)} & R
}
\]
where the vertical and horizontal arrows are log derivations and $h$ is a morphism of $B$-modules. By the universal property of symmetric algebras, $h$ extends uniquely to the symmetric algebra, $\Sym \Omega_{B/A}(M_B/M_A)$, making the following diagram commute
\[
\xymatrix{
\Sym\Omega_{B/A}(M_B/M_A) \ar@{.>}[dr]|-h &\\
B \ar[u]^{(j,d,\partial)} \ar[r]_{(D_0,D_1,\delta_1)} & R
}
\]
where the vertical and horizontal arrows are the higher log derivations of order $1$, and $h$ is a homomorphism of $B$-algebras. It suffices now to show that $h$ uniquely determines a morphism of monoids $h^\flat: M_\Omega\rightarrow M_R$, which make $(h,h^\flat)$ into morphism of log algebras. Note that the diagram
\[
\xymatrix{
M_R & \left(\bigoplus_{i=0}^\infty \Sym^i\Omega_{B/A}(M_B/M_A)\right)^*\ar[l]_<<<<{\alpha_R^{-1}\circ h}\\
M_B\ar[u]^{D_0^\flat} & B^*\ar[l]^{\alpha_B^{-1}}\ar[u]_j
}
\]
commutes since $\alpha_R^{-1}\circ h\circ j = \alpha_R^{-1}\circ D_0$ by the commutativity of the previous diagram and $\alpha_R^{-1}\circ D_0 = D_0^\flat\circ \alpha_B^{-1}$, when defined, since $D_0$ is a morphism of log algebras. Therefore by the universal property of amalgamated sums there exists a unique morphism $h^\flat$ such that $h^\flat \circ j^\flat = D_0^\flat$. By the First Universal Property, Theorem \ref{thm:func},
\[
\HS^1_{B/A}(M_B/M_A) \cong \bigoplus_{i=0}^\infty \Sym^i\Omega_{B/A}(M_B/M_A).
\]
\end{proof}

\begin{rem}
In particular by examining the proof we get that $\Omega_{B/A}(M_B/M_A)$ is just the degree $1$ graded component of $\HS^1_{B/A}(M_B/M_A)$.
\end{rem}

\begin{prop}\label{prop:fg}
Let $B$ be a finitely generated $A$ log algebra and $M_B$ finitely generated over $B^*$. Then $\HS^n_{B/A}(M_B/M_A)$ is a finitely generated log algebra over $A$.
\end{prop}

\begin{proof}
Since $M_B$ is finitely generated over $B^*$, there exists a finitely generated pre-log structure $\alpha:P\rightarrow M_B$ such that $M_B \cong P \oplus_{\alpha^{-1}(B)^*} B^*$. Let $p_1,\ldots,p_r$ denote the generating set of $P$. Similarly since $B$ is finitely generated over $A$, let $b_1,\ldots, b_s$ generate $B$ as an algebra over $A$. The ring, $\HS^n_{B/A}(M_B/M_A)$, is generated as an algebra over $A$ by elements of the form $d_i x, \partial_i m$ for $i=0, \ldots, n$ and $x\in B, m\in M_B$. By the $A$-linearity of $d_i$, Equation \ref{eq:sum}, $d_i x$ can be written as a sum of terms of the form $a d_i y$ where $a\in A$ and $y$ is a monomial in the $b_j$. By repeated application of Equation \ref{eq:product} the element $d_i y$ is a polynomial in the $d_v b_j$ for $v=0,\ldots, i$ and $j=1,\ldots, s$. Similarly let $q\in P$. Then $q=(p, u)$ where $p\in P$ is a monomial in the $p_j$ and $u\in B^*$ is a unit. However, $\partial_i u = \frac{d_i u}{u}$, which is a polynomial in the $d_v b_j$ over $A$ for $v=0,\ldots, i$ and $j=1,\ldots, s$ by our previous observation. By repeated application of Equation \ref{eq:monoidprod}, $\partial_i q$ is a polynomial in the $\partial_i p_k,\partial_i u$ for $i=1,\ldots, n$, $k=1,\ldots, r$. Therefore $\HS^n_{B/A}(M_B/M_A)$ is generated over $A$ by the $\partial_i p_k, d_i b_j$ for $i=1,\ldots n$, $j=1,\ldots, s$ and $k=1,\ldots, r$.
\end{proof}

\begin{rem}\label{rem:omegafg}
In the above setting $\Omega_{B/A}(M_B/M_A)$ is a finitely generated $B$-module. If $B$ is not fine this is not necessarily true.
\end{rem}

\begin{prop}\label{prop:welldef}
Let $f:A\rightarrow B$ be a morphism of log algebras. Then there exists a unique $A$-module homomorphism
\begin{equation}\label{eq:welldef}
d: \HS^n_{B/A}(M_B/M_A) \rightarrow \HS^{n+1}_{B/A}(M_B/M_A)
\end{equation}
satisfying
\begin{enumerate}
\item[i.] $d(uv) = ud(v)+vd(u)$ for all $u,v\in \HS^n_{B/A}(M_B/M_A)$
\item[ii] $d(u+v) = d(u)+d(v)$ for all $u,v\in \HS^n_{B/A}(M_B/M_A)$
\item[iii.] $d(d_i x) = (i+1)d_{i+1}x$ for all $x\in B$ and $i=0,\ldots,n$
\item[iv.] $d(\partial_i m) = (i+1)\partial_{i+1}m-\partial_1 m \partial_i m$ for all $m\in M_B$ and $i=1,\ldots,n$
\end{enumerate}
where the obvious inclusion maps have been omitted. By abuse of notation the author will refer to all such maps as $d$ regardless of $n$. Additionally define $d^m:\HS^n_{B/A}(M_B/M_A)\rightarrow \HS^{n+m}_{B/A}(M_B/M_A)$ by repeated application of $d$.
\end{prop}

\begin{proof}
First note that the map
\[
d: B[b^{(i)}, m^{(i)}]_{b\in B; m\in M_B; i=1,\ldots, n}\rightarrow B[b^{(i)}, m^{(i)}]_{b\in B; m\in M_B; i=1,\ldots, n+1}
\]
constructed identically as above is a well defined morphism of $A$-modules. The original map \ref{eq:welldef} is well defined if and only if $d(I^n_{B/A})\subset I^{n+1}_{B/A}$. It suffices to check the latter on a generating set. For $a\in A$ we have
\[
d(f(a)^{(i)}) = (i+1)f(a)^{(i+1)} \in I^{n+1}_{B/A}.
\]
Let $x,y\in B$. Then
\[
d\left((x+y)^{(i)}-x^{(i)}-y^{(i)}\right) = (i+1)\left((x+y)^{(i+1)}-x^{(i+1)}-y^{(i+1)}\right) \in I^{n+1}_{B/A}
\]
and
\begin{equation}\label{eq:sumtrick}
\begin{split}
d\left((xy)^{(i)}-\sum_{j+k=i} x^{(j)} y^{(j)}\right) &= (i+1)(xy)^{(i+1)}-\sum_{j+k=i}(j+1)x^{(j+1)}y^{(k)}-\sum_{j+k=i}(k+1)x^{(j)}y^{(k+1)}\\
&= (i+1)(xy)^{(i+1)}-\sum_{j+k=i+1} jx^{(j)}y^{(k)}-\sum_{j+k=i+1} kx^{(j)}y^{(k)}\\
&= (i+1)\left((xy)^{(i+1)}-\sum_{j+k=i+1}x^{(j)}y^{(k)}\right)\in I^{n+1}_{B/A}.
\end{split}
\end{equation}
Let $a\in M_A$. Then
\[
d(f^\flat(a)^{i}) = (i+1)f^\flat(a)^{(i+1)}-f^\flat(a)^{(1)}f^\flat(a)^{(i)}\in I^{n+1}_{B/A}.
\]
Let $m,p\in M_B$. Then
\[
\begin{split}
&d(\alpha(m)^{(i)}-\alpha(m)^{(0)}m^{(i)})\\
&=(i+1)\alpha(m)^{(i+1)}-\left(\alpha(m)^{(1)}m^{(i)}+\alpha(m)^{(0)}\left((i+1)m^{(i+1)}-m^{(1)}m^{(i)} \right)\right)\\
&=(i+1)\left(\alpha(m)^{(i+1)}-\alpha(m)^{(0)}m^{(i+1)}\right)-m^{(i)}\left(\alpha(m)^{(1)}-\alpha(m)^{(0)}m^{(1)}\right) \in I^{n+1}_{B/A}
\end{split}
\]
and
\begin{eqnarray*}
\lefteqn{d\left((m+p)^{(i)}-\sum_{j+k=i}m^{(j)}p^{(k)}\right) }\\
&&=(i+1)(m+p)^{(i+1)} - (m+p)^{(1)}(m+p)^{(i)}-\\ &&\sum_{j+k=i}\left(((j+1)m^{(j+1)} - m^{(1)}m^{(j)})p^{(k)} + m^{(j)}((k+1)p^{(k+1)} - p^{(1)}p^{(k)}\right).
\end{eqnarray*}
After a similar manipulation to \ref{eq:sumtrick} the above becomes
\[
(i+1)(m+p)^{(i+1)}-(m+p)^{(1)}(m+p)^{(i)}-(i+1)\sum_{j+k=i+1}m^{(j)}p^{(k)}+(m^{(1)}+p^{(1)})\sum_{j+k=i}m^{(j)}p^{(k)}
\]
and after grouping terms
\[
(i+1)\left((m+p)^{(i+1)}-\sum_{j+k=i+1}m^{(j)}p^{(k)}\right) +\left( (m^{(1)}+p^{(1)})\sum_{j+k=i}m^{(j)}p^{(k)}-(m+p)^{(1)}(m+p)^{(i)}\right),
\]
which is an element of $I^{n+1}_{B/A}$. Therefore $d(I^n_{B/A})\subset I^{n+1}_{B/A}$ and the morphism is well defined.
\end{proof}

\begin{thm}\label{thm:gendiff}
Let $B$ be an $A$ log algebra, $A$ a $k$-algebra with char(k)=0. Suppose $\HS^n_{B/A}(M_B/M_A)$ is generated by $\{\omega_i\}_{i\in I}$ over $B$ as an algebra for some $n\geq 1$. Then $\HS^{n+1}_{B/A}(M_B/M_A)$ is generated by $\{\omega_i,d(\omega_i)\}_{i\in I}$ over $B$ as an algebra, where $d$ is defined in Proposition \ref{prop:welldef}.
\end{thm}

\begin{proof} Note that $\HS^{n+1}_{B/A}(M_B/M_A)$ is generated over $\HS^n_{B/A}(M_B/M_A)$ by elements of the form $d_{n+1}b$ and $\partial_{n+1}m$ where $b\in B$ and $m\in M_B$. Starting with the former observe that $d_n b\in \HS^n_{B/A}(M_B/M_A)$, so by hypothesis
\[
d_n b = \sum_{j=1}^p b_j\alpha_j
\]
for some $b_j\in B$ and $\alpha_j$ monomials in the $\omega_i$. By Proposition \ref{prop:welldef}
\[
(n+1)d_{n+1}b = d(d_nb) = \sum_{j=1}^p\left(b_jd(\alpha_j)+\alpha_jd_1 b_j\right).
\]
By hypothesis $d_1 b_j\in \HS^n_{B/A}(M_B/M_A)$ and by repeated application of \ref{prop:welldef} part i, $d(\alpha_j)\in B[\omega_i, d(\omega_i)]_{i\in I}$. Since $n+1\neq 0$, $d_{n+1}b\in B[\omega_i,d(\omega_i)]_{i\in I}$. For the latter case note that
\[
\partial_{n}m = \sum_{j=1}^r m_j \alpha_j
\]
for some $m_j\in B$ and $\alpha_j$ monomials in the $\omega_i$. By Proposition \ref{prop:welldef}
\[
(n+1)\partial_{n+1}m-\partial_1 m\partial_n m = d(\partial_n m) = \sum_{j=1}^r\left(\alpha_jd_1 m_j+m_j d(\alpha_j)\right).
\]
By hypothesis $d_1 m_j, \partial_1 m, \partial_n m \in \HS^n_{B/A}(M_B/M_A)$ and by repeated application of \ref{prop:welldef} part i, $d(\alpha_j)\in B[\omega_i,d(\omega_i)]_{i\in I}$. Since $n+1 \neq 0$, $\partial_{n+1}m \in B[\omega_1,d(\omega_1),\ldots, \omega_r, d(\omega_r)]$. Thus $\HS^{n+1}_{B/A}(M_B/M_A)$ is generated by $\{\omega_i,d(\omega_i)\}$ over $B$.
\end{proof}

\begin{rem} The condition that the characteristic of the field be $0$ is necessary in general. Let $A = \mathbb{F}_2$, $B = \mathbb{F}_2[x]$, both with trivial log structure. Then $\HS^1_{B/A}(M_B/M_A) = \mathbb{F}_2[x, d_1 x]$. However, $d(d_1 x) = 0$ and $\HS^2_{B/A}(M_B/M_A) = \mathbb{F}_2[x,d_1 x,d_2 x]$.
\end{rem}

\begin{example}\label{ex:gendifff}
In characteristic $0$ we are now able to recover what might be considered a more standard definition of derivative. Consider $\ln f(z)$ where $f(z)$ is some analytic function and a branch of the natural logarithm has been fixed. Then
\[
\frac{d}{dz} \ln f(z) = \frac{f'(z)}{f(z)}
\]
and
\[
\frac{d^2}{dz^2} \ln f(z) = \frac{f''(z)}{f(z)} - \left(\frac{f'(z)}{f(z)}\right)^2.
\]
On the other hand let $\alpha:M\rightarrow \C[x]$ be a log structure with some non-constant element $f\in M$. Then
\[
\partial_1 f = \frac{d_1 \alpha(f)}{\alpha(f)}
\]
which agrees in spirit at least with the standard complex case. However,
\[
\partial_2 f = \frac{d_2 \alpha(f)}{\alpha(f)},
\]
which is vastly different. If instead we consider
\[
d\partial_1 f = 2\frac{d_2 \alpha(f)}{\alpha(f)}-\left(\frac{d_1\alpha(f)}{\alpha(f)}\right)^2
\]
we get something much closer. Recalling that the $d_i$ are in fact divided differentials, the two expressions agree identically.
\end{example}

\section{Constructions and Properties}
The next goal is to actually compute some of these log Hasse-Schmidt rings. Unfortunately definition \ref{def:hsdef} gives a construction which is not easy to work with in practice. In the case of Hasse-Schmidt rings for quotients of polynomial algebras there is a very straightforward method of computing the resulting Hasse-Schmidt ring, see \cite{vojt1} for instance.

The method described in \cite{vojt1} is an application of the second fundamental exact sequence for differential forms to the case of Hasse-Schmidt rings. This does further generalize to the case of log Hasse-Schmidt rings, but some extra conditions must be imposed. For instance consider the following example

\begin{example}\label{ex:bestexample}
Let $B=k[x,y]$, $C=k[x,y]/(x^2-y^3)$ be $k$ algebras, $\character(k)\neq 2,3$. From \cite{vojt1} we know that
\[
\HS^1_{C/k}\cong \HS^1_{B/k}/(x^2-y^3,d(x^2-y^3)) \cong k[x,y,dx,dy]/(x^2-y^3, 2xdx-3y^2dy)
\]
Now define $M_B$ to be the minimal sub-log structure of $B^\times$ containing $x$ and $y$ and define $M_C$ to be the minimal sub-log structure of $C^\times$ containing $x$ and $y$. While we do not yet have the machinery to prove it, we will soon be able to see
\[
\begin{split}
\HS^1_{B/k}(M_B/k^*)/(x^2-y^3,d(x^2-y^3)) &\cong
k[x,y,\partial x,\partial y]/(x^2-y^3, 2x^2\partial x-3y^3\partial y)\\
&\cong k[x,y,\partial x,\partial y]/(x^2-y^3, x^2(2\partial x-3\partial y)),
\end{split}
\]
which is not an integral domain. This may not seem that significant, but the actual log Hasse-Schmidt ring should be
\[
\HS^1_{C/k}(M_C/k^*) \cong k[x,y,\partial x,\partial y]/(x^2-y^3,2\partial x-3\partial y) \cong k[x,y,\partial x]/(x^2-y^3),
\]
which is an integral domain and is well behaved. Although this is not a rigorous proof that this is the answer, one can at least see that $2\partial x-3\partial y$ has to be in the quotient. In $M_C$, $x^2=y^3$ since it is a sub-monoid of $C^\times$ and $x^2=y^3$ in $C^\times$, so $\partial(x^2)=\partial(y^3)$, which implies $2\partial x-3\partial y=0$. Therefore the desired relation holds.
\end{example}

Obviously the problem here is that there are possibly relations on $M_C$ which don't come from $M_B$. There are a couple of approaches to fixing this. One is to simply quotient out by all additional relations on $M_C$. This is cumbersome in general. The second is to simply restrict what is allowed for $M_C$. In particular, we could have given $M_C$ the log structure generated by $x,y$ with no relations. It turns out this works and is the same as asking the morphism $B\rightarrow C$ to be strict. This will be a sufficient condition for the second fundamental exact sequence to hold.

This answer is still not satisfactory though. The ring $C$ in this example is actually a monoid algebra, and the log structure described in the example comes from the corresponding monoid. If $C$ were to have any sort of natural log structure, the $M_C$ described above should be it. There is fortunately another solution and while it will not work in every case, it will give us a different tool for constructing log Hasse-Schmidt rings.

Let $S\subset \HS^1_{B/k}(B^*/k^*)/(x^2-y^3)$ be the multiplicative set generated by $x$ and $y$. Then
\[
\begin{split}
S^{-1}\HS^1_{B/k}(B^*/k^*)/(x^2-y^3,d(x^2-y^3)) &\cong k[x,y,\partial x,\partial y]_{xy}/(x^2-y^3,x^2(2\partial x-3\partial y))\\
&\cong k[x,y,\partial x,\partial y]_{xy}/(x^2-y^3,(2\partial x-3\partial y))
\end{split}
\]
and we can construct $\HS^1_{C/k}(M_C/k^*)$ as the subring generated by $x,y,\partial x,\partial y$. Of course there will have to be some conditions put in place and all of this will have to be made formal, but it is a possible approach in many instances.

\begin{thm}[Second Fundamental Exact Sequence]\label{thm:construction}
Let $A\rightarrow B\rightarrow C$ be a sequence of log algebras. Additionally suppose $B\rightarrow C$ is surjective and strict in the category of log algebras with kernel $I$. Define $J$ to be the ideal
\[
J := (d_i x)_{i=0,\ldots n; x\in I} \subseteq \HS^n_{B/A}(M_B/M_A).
\]
Then
\[
0\rightarrow J\rightarrow \HS^n_{B/A}(M_B/M_A)\rightarrow \HS^n_{C/A}(M_C/M_A)\rightarrow 0
\]
is an exact sequence of $B$-modules.
\end{thm}

\begin{proof}
Let $R$ be an arbitrary log algebra over $A$ and observe that
\begin{equation} \label{eq:cpone}
\Hom_A(C,R[t]/t^{n+1})\rightarrow \{\phi\in\Hom_A(B, R[t]/t^{n+1}): \phi(I)=0\}
\end{equation}
is a bijection. To see this let $(\phi,\phi^\flat) \in \{\phi\in\Hom_A(B, R[t]/t^{n+1}): \phi(I)=0\}$. Then it follows that $\phi$ factors through $B/I\cong C$, which induces a unique morphism of rings $C\rightarrow R[t]/t^{n+1}$. For the log structure note that $\phi$ defines a morphism $\phi|_{C^*}:C^*\rightarrow \widehat{M}^n_R$, which together with $\phi^\flat: M_B\rightarrow \widehat{M}^n_R$ give a unique morphism $M_C\cong M_B\oplus_{B^*} C^*\rightarrow \widehat{M}^n_R$ by the universal property of amalgamated sums. Therefore the above map on sets \ref{eq:cpone} is a natural bijection. To complete the proof we need a bijection
\[
\Hom_A(\HS^n_{B/A}(M_B/M_A)/J, R)\rightarrow\{\phi\in\Hom_A(B, R[t]/t^{n+1}): \phi(I)=0\},
\]
where the log structure on $\HS^n_{B/A}(M_B/M_A)/J$ is determined by the log structure coming from $\HS^n_{B/A}(M_B/M_A)$. By \ref{cor:functor} we have a bijection
\[
\Hom_A(\HS^n_{B/A}(M_B/M_A), R)\rightarrow \Hom_A(B, R[t]/t^{n+1}).
\]
By the same Corollary, $\phi(I)=0$ if and only if the corresponding morphism $\HS^n_{B/A}(M_B/M_A)\rightarrow R$ sends $\{d_i x\}_{i=0,\ldots n; x\in I}$ to $0$, which gives the desired bijection. Combining everything gives
\[
\Hom_A(\HS^n_{B/A}(M_B/M_A)/J, R)\cong \Hom_A(\HS^n_{C/A}(M_C/M_A), R)
\]
for every $R$. Therefore
\[
\HS^n_{B/A}(M_B/M_A)/J\cong \HS^n_{C/A}(M_C/M_A),
\]
and the desired exact sequence must hold.
\end{proof}

\begin{cor}
In the above setting let $I$ be generated by $\{\alpha_j\}_{j\in J}$. Then $J$ is generated by $\{d_i\alpha_j\}_{i=0,\ldots,n;j\in J}$.
\end{cor}

\begin{proof}
Let $x\in I$. Then
\[
x = \sum_{j\in J} x_j \alpha_j,
\]
where $x_j\in B$ and all but finitely many $x_j=0$, so
\[
d_ix_j = \sum_{j\in J} d_i(x_j\alpha_j) = \sum_{j\in J} \sum_{k+l=i} d_k x_j d_l\alpha_j \in \left(d_i\alpha_j\right)_{i=0,\ldots,n;j\in J}.
\]
\end{proof}

Although we will not be needing it, a version of the first fundamental exact sequence can also be stated for log Hasse-Schmidt rings.

\begin{thm}[First Fundamental Exact Sequence] Let $A\rightarrow B\rightarrow C$ be a sequence of log algebras. Then the sequence
\[
0\rightarrow \HS^n_{B/A}(M_B/M_A)^+\HS^n_{C/A}(M_C/M_A)\rightarrow \HS^n_{C/A}(M_C/M_A)\rightarrow \HS^n_{C/B}(M_C/M_B)\rightarrow 0,
\]
where $\HS^n_{B/A}(M_B/M_A)^+$ is the ideal of elements with positive degree, is an exact sequence of $C$-modules.
\end{thm}

\begin{proof}
The first term is an ideal and therefore the sequence is exact on the left. The inclusion of ideals $I^n_{C/A}\subset I^n_{C/B}$ makes the sequence exact on the right.
By Definition \ref{def:hsdef}, $I^n_{C/B}\setminus I^n_{C/A}$ forms a generating set of $\HS^n_{B/A}(M_B/M_A)^+$. Therefore the sequence is exact in the middle.
\end{proof}

\begin{prop}\label{prop:boringpoly}
Let $A$ be a log ring and $B=A[x_i]_{i\in I}$ a polynomial ring with log structure $M_A$. Then
\[
\HS^n_{B/A}(M_B/M_A)\cong A[d_q x_i]_{q=0,\ldots,n;i\in I}
\]
\end{prop}

\begin{proof}
Let $R$ be an $A$ log algebra. Then any homomorphism of log rings $\phi: A[x_i]_{i\in I}\rightarrow R[t]/t^{n+1}$ is uniquely determined by the image of the $x_i$ which is free. Therefore
\[
\Hom_A(A[x_i]_{i\in I}, R[t]/t^{n+1})\cong\Hom_A(A[d_q x_i]_{q=0,\ldots,n;i\in I}, R)
\]
and by \ref{cor:functor}
\[
\HS^n_{B/A}(M_B/M_A)\cong A[d_q x_i]_{q=0,\ldots,n;i\in I}
\]
\end{proof}

\begin{prop}\label{prop:loc1}
Let $f: A \rightarrow B$ be a morphism of log algebras and $S\subset B$ a multiplicative subset. Then the natural map
\[
B\rightarrow S^{-1}B
\]
induces an isomorphism
\begin{equation}\label{eq:localization}
S^{-1}\HS^n_{B/A}(M_B/M_A)\rightarrow \HS^n_{S^{-1}B/A}(S^{-1}M_B/M_A)
\end{equation}
of rings.
\end{prop}

\begin{proof}
Let $R$ be any log algebra over $A$ and let $f:\HS^n_{B/A}(M_B/M_A)\rightarrow R$ be a morphism of log algebras over $A$ such that $f(s)\in R^*$ for all $s\in S$. By Corollary \ref{cor:functor} there exists a unique morphism of log algebras $B\rightarrow R[t]/t^{n+1}$ which also sends $s$ to a unit. By the universal property of localized log algebras there exists a unique morphism $S^{-1}B\rightarrow R[t]/t^{n+1}$, which by Corollary \ref{cor:functor} corresponds to a unique morphism $\HS^n_{S^{-1}B/A}(S^{-1}M_B/M_A)\rightarrow R$. Therefore every morphism $f:\HS^n_{B/A}(M_B/M_A)\rightarrow R$ with $f(s)\in R^*$ for all $s\in S$ factors uniquely through $\HS^n_{B/A}(M_B/M_A)\rightarrow \HS^n_{S^{-1}B/A}(S^{-1}M_B/M_A)$. Since $R$ was arbitrary, by Proposition \ref{quotientuniv}
\[
\HS^n_{S^{-1}B/A}(S^{-1}M_B/M_A) \cong S^{-1}\HS^n_{B/A}(M_B/M_A)
\]
\end{proof}

The following corollary will allow for much simpler constructions of log Hasse-Schmidt rings in many instances.

\begin{cor}
Let $B$ be a log algebra over $A$ where $B$ is an integral domain and $A$ has trivial log structure. Suppose additionally that $\alpha_B:M_B\rightarrow B$ is injective, i.e., $M_B$ is a submonoid of $B$. Then there exists an inclusion
\[
\HS^n_{B/A}(M_B/A^*) \hookrightarrow M_B^{-1}\HS^n_{B/A}
\]
\end{cor}

\begin{proof}
By hypothesis $B$ is integral and thus $\HS^n_{B/A}(M_B/A^*) \hookrightarrow M_B^{-1}\HS^n_{B/A}(M_B/A^*)$ is an inclusion. By Proposition \ref{prop:loc1} applied twice,
\[
\begin{split}
M_B^{-1}\HS^n_{B/A}(M_B/A^*) &\cong \HS^n_{M_B^{-1}B/A}(M_B^{-1}M_B/A^*)\\
&\cong \HS^n_{M^{-1}_B B/A}((M^{-1}_B B)^*/A^*)\\
&\cong  M_B^{-1}\HS^n_{B/A}
\end{split}
\]
\end{proof}

\begin{example}\label{ex:calcd}
Let $B=k[x,y]$, $C = k[x,y]/(x^2-y^3)$ be $k$-algebras with $\character{k} \neq 2,3$. Let all rings have trivial log structure. By Proposition \ref{prop:boringpoly}
\[
\HS^1_{B/k} \cong k[x,y,d_1x,d_1y]
\]
and by Theorem \ref{thm:construction}
\[
\HS^1_{C/k} \cong k[x,y,d_1x, d_1y]/(x^2-y^3, 2xd_1x-3y^2d_1y).
\]
Define $M_C$ to be the log structure generated by $x$ and $y$. Then
\[
\HS^1_{C/k^*}(M_C/k^*)\hookrightarrow k[x,y,d_1x, d_1y]_{xy}/(x^2-y^3, 2xd_1x-3y^2d_1y)
\]
and
\[
k[x,y,d_1x, d_1y]_{xy}/(x^2-y^3, 2xd_1x-3y^2d_1y) \cong k[x,y,d_1 x, d_1y]_{xy}/(x^2-y^3, 2\frac{d_1 x}{x}-3\frac{d_1 y}{y})
\]
Since $x$ and $y$ generate the log structure of $C$ and $x$ and $y$ generate $C$ over $k$,
\[
\begin{split}
\HS^1_{C/k^*}(M_C/k^*) &\cong k[x,y,\partial_1 x, \partial_1 y]/(x^2-y^3, 2\partial_1 x-3\partial_1 y)\\
&\cong k[x,y,\partial_1 x]/(x^2-y^3)
\end{split}
\]
\end{example}

\begin{prop}\label{prop:locbase} Let $f:A\rightarrow B$ be a morphism of log algebras, $S\subset A$ be a multiplicative subset, and suppose $f$ factors through the natural map $A\rightarrow S^{-1}A$. Then the identity map on $B[b^{(i)}, m^{(i)}]_{b\in B; m\in M_B; i=1,\ldots, n}$ induces an isomorphism
\[
\HS^n_{B/A}(M_B/M_A)\rightarrow \HS^n_{B/S^{-1}A}(M_B/M_{S^{-1}(A)})
\]
\end{prop}

\begin{proof}
Let $f':S^{-1}A\rightarrow B$ be the factored morphism, $a\in A$, $s\in S$ and $k\in\{1,\ldots, m\}$. Then
\[
0 = d_kf(a) = d_k(f(s)f'(s^{-1}a)) = \sum_{i+j=k} d_if(s) d_j f'(s^{-1}a) = f(s)d_kf'(s^{-1}a)
\]
in $\HS^n_{B/A}(M_B/M_A)$, by induction on $k$. However, $f(s)$ is invertible on $B$, so $d_kf'(s^{-1}a)$ vanishes in $\HS^n_{B/A}(M_B/M_A)$. An identical argument follows for $a\in M_A$ and $s\in S$. Thus the ideals are equal and the morphism induces an isomorphism.
\end{proof}

\begin{lem}\label{stalksheaf}
Let $f:X\rightarrow Y$ be a morphism of quasi-coherent affine log schemes, with $X = \spec B$ and $Y=\spec A$ such that $M_X\cong \widetilde{M_X(X)}$ and $M_Y\cong \widetilde{M_Y(Y)}$. Then for every $p\in X$ there exists a morphism
\[
\HS^n_{B/A}(M_X/M_Y)\rightarrow \HS^n_{B_p/A_{f(p)}}(M_{X,p},M_{Y,f(p)})
\]
uniquely determined by the localization map $B\rightarrow B_p$.
\end{lem}

\begin{proof}
Let $M_B = M_X(X)$, $M_A = M_Y(Y)$, $\alpha_B:M_B\rightarrow B$ and $\alpha_A:M_A\rightarrow A$ be the monoids and morphisms making $B$ and $A$ into log algebras. Then localization $B\rightarrow B_p$ induces a morphism
\[
\HS^n_{B/A}(M_B/M_A)\rightarrow \HS^n_{B_p/A}((M_B)_p, M_A).
\]
However, by Proposition \ref{prop:locbase}
\[
\HS^n_{B_p/A}((M_B)_p, M_A)\cong \HS^n_{B_p/A_{f(p)}}((M_B)_p, (M_A)_{f(p)})
\]
and by Proposition \ref{monoidstalk} $M_{X,p} \cong (M_B)_{p}$, $M_{Y,f(p)} \cong (M_A)_{f(p)}$ giving the desired morphism.
\end{proof}

Now that we have the tools necessary,
we can construct the sheaf of higher log differentials. The useful case will be on log schemes with quasi-coherent log structure, but the sheaf can be defined for arbitrary log schemes.

\begin{mydef}
Let $f:X\rightarrow Y$ be a morphism of log schemes. Define the \textit{log Hasse-Schmidt sheaf of order n}, denoted $\HS^n_{X/Y}(M_X/M_Y)$, on each open set $U\subset X$ to be the set of functions
\[
s:U\rightarrow \coprod_{p\in U} \HS^n_{\Oc_{X,p}/\Oc_{Y,f(p)}}(M_{X,p}/M_{Y,f(p)})
\]
such that
\begin{enumerate}
\item[i.] for every $p\in U$, $s(p)\in \HS^n_{\Oc_{X,p}/\Oc_{Y,f(p)}}(M_{X,p}/M_{Y,f(p)})$; and

\item[ii.] for every $p\in U$ there exists an open neighborhood $V\subset U$ containing $p$, an open affine neighborhood $W\subset Y$ containing $f(p)$, and a $t\in \HS^n_{\Oc_X(V)/\Oc_Y(W)}(M_X(V)/M_Y(W))$ such that $V\subset f^{-1}(W)$ and the image of $t$ under the morphism
    \[
    \HS^n_{\Oc_X(V)/\Oc_Y(W)}(M_X(V)/M_Y(W))\rightarrow \HS^n_{\Oc_{X,p}/\Oc_{Y,f(p)}}(M_{X,p}/M_{Y,f(p)})
    \]
    defined in Lemma \ref{stalksheaf} is equal to $s(p)$ for every $p\in V$.

\end{enumerate}
If the log structures $M_X$ and $M_Y$ are trivial the notation will be shortened to $\HS^n_{X/Y}$.
\end{mydef}

\begin{thm}\label{qcsheaf}
Let $f:X\rightarrow Y$ be a morphism of quasi-coherent log schemes. Then $\HS^n_{X/Y}(M_X/M_Y)$ is a quasi-coherent sheaf of $\Oc_X$-algebras.
\end{thm}

\begin{proof}
By construction $\HS^n_{X/Y}(M_X/M_Y)$ is a sheaf. Let $p\in X$ be arbitrary. By the quasi-coherence of $X$ and $Y$ there exist open affine sets $\spec A\subset Y$ containing $f(p)$ and $\spec B \subset f^{-1}(\spec A)$ containing $p$ such that $M_Y|_{\spec A} \cong \widetilde{P_A}$ and $M_X|_{\spec B} \cong \widetilde{P_B}$ for some log structures $\alpha_A:P_A\rightarrow A$ and $\alpha_B:P_B\rightarrow B$. We may assume $\Gamma(\spec A, \widetilde{P_A}) = P_A$ and $\Gamma(\spec B, \widetilde{P_B}) = P_B$ after a suitable substitution. For each $g\in B$ we have an open affine $D(g)$, and
\[
\Gamma(D(g), \widetilde{\HS}^n_{B/A}(P_B/P_A)) \cong \HS^n_{B_g/A}((P_B)_g/P_A),
\]
where the isomorphism comes from the fact that $\widetilde{\HS}^n_{B/A}(P_B/P_A)$ is quasi-coherent, and Proposition \ref{prop:loc1}. So the natural map
\[
\HS^n_{B_g/A}((P_B)_g/P_A)\rightarrow\HS^n_{B_g/A}(\widetilde{(P_B)_g}/\widetilde{P_A})
\]
defines a morphism
\[
\Gamma(D(g), \widetilde{\HS}^n_{B/A}(P_B/P_A))\rightarrow \Gamma(D(g),\HS^n_{B_g/A}(\widetilde{(P_B)_g}/\widetilde{P_A}))
\]
for every $D(g)$. Since this construction is compatible with localization and the $D(g)$ form a basis for the topology of $\spec B$ we get a morphism of sheaves
\[
\phi: \widetilde{\HS}^n_{B/A}(P_B/P_A)\rightarrow \HS^n_{X/Y}(M_X/M_Y)
\]
on $\spec B$. For any $q\in \spec B$ the stalk of $\HS^n_{X/Y}(M_X/M_Y)$ at $q$ is
\[
\HS^n_{\Oc_{X,q}/\Oc_{Y,f(q)}}(M_{X,q}/M_{Y,f(q)}),
\]
but by Proposition \ref{monoidstalk} and Definition \ref{def:monoidlocal}
\[
\HS^n_{\Oc_{X,q}/\Oc_{Y,f(q)}}(M_{X,q}/M_{Y,f(q)}) \cong \HS^n_{B_q/A_{f(q)}}((P_B)_q, (P_A)_{f(q)}).
\]
By Propositions \ref{prop:loc1} and \ref{prop:locbase}
\[
\left(\widetilde{\HS}^n_{B/A}(P_B/P_A)\right)_q \cong \HS^n_{B_q/A_{f(q)}}((P_B)_q, (P_A)_{f(q)}),
\]
so $\phi$ induces an isomorphism of stalks. Since $q$ was arbitrary $\phi$ is an isomorphism of sheaves on $\spec B$, and since $p$ was arbitrary $\HS^n_{X/Y}(M_X/M_Y)$ is quasi-coherent.
\end{proof}

\begin{prop}\label{coherentdiff}
Let $f:X\rightarrow Y$ be a morphism of noetherian coherent log schemes, and $f$ locally of finite type. Then $\Omega_{X/Y}(M_X/M_Y)$ is a coherent sheaf of $\Oc_X$-modules.
\end{prop}

\begin{proof}
Let $p\in X$ and define $A, B, P_A, P_B$ as they are in the proof of Theorem \ref{qcsheaf}. Then we have an isomorphism of sheaves
\[
\HS^n_{X/Y}(M_X/M_Y)|_{\spec B}\cong \widetilde{\HS}^n_{B/A}(P_B/P_A).
\]
Since $f$ is locally of finite type, $B$ is a finitely generated $A$-algebra. Moreover, because $X$ is coherent, $P_B$ is finitely generated over $B^*$. By Proposition \ref{prop:fg} and Remark \ref{rem:omegafg}, $\Omega_{B/A}(P_B/P_A)$ is a finitely generated $B$-module. Since $p$ was arbitrary $\Omega_{X/Y}(M_X/M_Y)$ is coherent.
\end{proof}

Finally log jet spaces can be defined.

\begin{mydef}
Let $f:X\rightarrow Y$ be a morphism of quasi-coherent log schemes and define the \textit{log jet space of order} $n$ \textit{of} $X$ \textit{over} $Y$ to be the scheme
\[
J^n_{X/Y}(M_X/M_Y) := \sheafspec\HS^n_{X/Y}(M_X/M_Y)
\]
over $X$. If the log structures $M_X$ and $M_Y$ are trivial, the notation will be shortened to $J^n_{X/Y}$.
\end{mydef}

\section{\'{E}tale Morphisms}

Building upon the previous section, \'{e}tale morphisms will give another method to construct log Hasse-Schmidt rings. Unfortunately the definition of formally \'{e}tale in the category of log algebras differs from the traditional one in the category of rings. For a full treatment see \cite{ogus06}, but we can at least motivate some of the reasoning with an example.

\begin{example}
Let $C$ be as defined in Example \ref{ex:bestexample}. By Example \ref{ex:calcd} we know that $\Omega_{C/k}(M_C/k^*)$ is free and generated by a single element, $\partial_1 x$. In the case of rings this would typically indicate that $C$ is smooth. However, consider the diagram
\[
\xymatrix{
C \ar[r]^f & k[t]/t^2\\
k\ar[u]\ar[r] & k[t]/t^3\ar[u]
}
\]
where $k[t]/t^2$ and $k[t]/t^3$ have log structures $(k[t]/t^2)^\times$ and $(k[t]/t^3)^\times$ respectively and $k$ has the trivial log structure. Define $f$ as $f(x)=t$ and $f(y)=t$, which induces a surjective morphism on log structures. Suppose there exists some $u:C\rightarrow k[t]/t^3$ making the diagram commute. Then $u(x) = t+a t^2, u(y)=t+bt^2$ where $a,b\in k$. However, $u(0)=u(x)^2-u(y)^3 = t^2\neq 0$, so therefore no such $u$ can exist.

This should indicate that this scheme is simply not smooth, but its logarithmic tangent space at any point is $1$-dimensional, which in the category of schemes would indicate smoothness. Therefore these two definitions do not initially appear to be compatible. However, this problem can be remedied by adjusting the log structures in this example. If we require that $k[t]/t^3\rightarrow k[t]/t^2$ be a strict morphism of log algebras then the log structure on $k[t]/t^2$ is given by $(k[t]/t^2)^*\oplus_{(k[t]/t^3)^*} (k[t]/t^3)^\times$ and the $f$ constructed above is no longer well defined.

Strictness in and of itself is not a bad requirement, but unfortunately it is not enough in general. The morphism $k[x]\rightarrow C$ should be \'{e}tale, in particular unramified, when the log structure on $k[x]$ is the minimal sub-log structure of $(k[x])^\times$ containing $x$. This morphism induces a homeomorphism of topological spaces associated to the corresponding varieties, in addition it induces an isomorphism on their sheaves of log differentials. It should be reasonable that such a map is \'{e}tale. Consider the diagram
\[
\xymatrix{
C\ar[r]^f &k[t]/t^2\\
k[x]\ar[u]\ar[r]_h & k[t]/t^5\ar[u]_g
}
\]
where $k[t]/t^5$ has log structure $(k[t]/t^5)^\times$ and $k[t]/t^2$ has log structure $(k[t]/t^2)^*\oplus_{(k[t]/t^5)^*} (k[t]/t^5)^\times$. Additionally define $h(x) = f(x) = f(y) =0$ and $f^\flat(x) = (1,0)$, $f^\flat(y) = (1,t^2)$. For every $a\in k$ there exists a $u_a:C\rightarrow k[t]/t^5$ such that $u_a(x) = 0$, $u_a(y) = t^2+at^4 = (1+at^2)t^2$. Note that in this case the morphism of log structures is entirely determined by $u_a$. As a map of rings, $g\circ u_a = f$. In $(k[t]/t^2)^*\oplus_{(k[t]/t^5)^*} (k[t]/t^5)^\times$
\[
(1,(1+at^2)t^2) = (1+at^2, t^2) = (1, t^2)
\]
and therefore $g^\flat\circ u_a^\flat = f^\flat$. However, $a$ was arbitrary and $k[x]\rightarrow C$ should be unramified.

Identifying the source of the problem in this case is a little more difficult. Definitely part of the problem is that $u_a^\flat(y) \neq 0$ but $f^\flat(y) = 0$. This is where some wiggle room is given for the various $u_a$. The standard method of preventing this from happening is requiring that the group $1+\ker(g)$ acts freely on the log structure of $k[t]/t^5$. In this particular case it would mean that $t^2+at^4$ cannot be in the log structure as $(1+t^3)\cdot (t^2+at^4) = t^2+at^4$.
\end{example}

\begin{mydef}
A \textit{log thickening} is a strict surjective morphism $i:T\rightarrow S$ of log algebras such that $\ker(i)$ is nilpotent and the subgroup $1+\ker(i)\subset T^*$ acts freely on $M_T$.
\end{mydef}

\begin{mydef}
A morphism of log algebras $B\rightarrow C$ is \textit{formally log smooth} (respectively \textit{log unramified}, respectively \textit{log \'{e}tale}) if for every commutative diagram
\[
\xymatrix{
C \ar[r]^v & S\\
B \ar[u] \ar[r] &T \ar[u]_i
}
\]
where $i:T\rightarrow S$ is a log thickening, there exists at least (respectively at most, respectively exactly) one morphism $u:C\rightarrow T$ such that $i\circ u = v$.
\end{mydef}

\begin{thm}
Let $f: B\rightarrow C$ be an \'{e}tale morphism of log algebras over $A$. Then
\[
\HS^n_{B/A}(M_B/M_A)\otimes_B C \cong \HS^n_{C/A}(M_C/M_A),
\]
where the log structure $M$ on $\HS^n_{B/A}\otimes_B C$ is defined by the pre-log structure $\alpha_C:M_C\rightarrow \HS^n_{B/A}(M_B/M_C)\otimes_B C$.
\end{thm}

\begin{proof}
Let $R$ be a log algebra over $A$ and $\beta\in\Hom_A(\HS^n_{B/A}(M_B/M_A)\otimes_B C, R)$. Then by the universal properties of tensor product and amalgamated sum there exists a unique commutative diagram of $A$ log algebras
\[
\xymatrix{
C\ar[r] & R\\
B\ar[u]\ar[r] & \HS^n_{B/A}(M_B/M_A)\ar[u]
}
\]
By the universal property of log Hasse-Schmidt rings there exists a unique commutative diagram of $A$ log algebras
\[
\xymatrix{
C\ar[r] & R\\
B\ar[u]\ar[r] & R[t]/t^{n+1} \ar[u]
}
\]
Since $B\rightarrow C$ is \'{e}tale and $R[t]/t^{n+1}\rightarrow R$ is a log thickening there exists a unique morphism $C\rightarrow R[t]/t^{n+1}$ making the diagram commute. Again by the universal property of log Hasse-Schmidt rings there exists a unique morphism $\HS^n_{C/A}(M_C/M_A)\rightarrow R$ of $A$ log algebras. This construction is reversible, so $\Hom_A(\HS^n_{B/A}(M_B/M_A)\otimes_B C, R)\cong \Hom_A(\HS^n_{C/A}(M_C/M_A), R)$. Since $R$ was an arbitrary log algebra over $A$,
\[
\HS^n_{B/A}(M_B/M_A)\otimes_B C \cong \HS^n_{C/A}(M_C/M_A).
\]
\end{proof}

\section{Main Result}

\begin{mydef}
Let $X$ be a scheme, $p$ a closed point, $D$ an effective Cartier divisor on $X$ and $s$ a regular function representing $D$ in some neighborhood of $p$. Define the \textit{multiplicity of $D$ at $p$}, denoted $\mult_p D$, to be the largest integer $n$ such that $s\in\textbf{m}^n_p$.
\end{mydef}

\begin{prop}
$\mult_p D$ is independent of the choice of $s$.
\end{prop}

\begin{proof}
Let $s$ and $s'$ both represent $D$ in $\Oc_{X,p}$. Then by definition $s = u s'$ where $u\in \Oc_{X,p}^*$, therefore $s\in \textbf{m}_p^n$ if and only if $s' \in\textbf{m}_p^n$.
\end{proof}

\begin{prop}\label{prop:loc}
Let $X$ be an equidimensional scheme of dimension $q$ over an algebraically closed field $k$. Then for any nonsingular closed point $p$ the completion of the local ring $\widehat{\Oc}_{X,p}$ is isomorphic to $k[[x_1,\ldots,x_q]]$.
\end{prop}

\begin{proof}
This follows directly from \cite{hart77} Theorem I.5.5A.
\end{proof}

\begin{prop}\label{prop:smoothcover}
Let $X$ be a smooth scheme of dimension $q$ over a field $k$. Then there exists an open cover $\{U_i\}$ such that
\[
J^n_{U_i/k} \cong U_i\times_k \A^{qn}_k.
\]
\end{prop}

\begin{proof}
This is a special case of Proposition 5.10 in \cite{vojt1}.
\end{proof}

\begin{thm}\label{thm:factorfibers}
Let $X$ be a scheme of dimension $q$ over an algebraically closed field $k$, $p\in X$ a nonsingular closed point, $D$ an effective Cartier divisor on $X$ locally represented by $s\in\Oc_{X,p}$. Then $\mult_p D \geq n+1$ if and only if $\{d_i s\}_{i=0}^n\subset \textbf{m}_p\HS^n_{\Oc_{X,p}/k}$.
\end{thm}

\begin{proof}
Suppose $\mult_p D\geq n+1$. Let $h:\Oc_{X,p}\rightarrow k[t]/t^{n+1}$ be any $k$-algebra homomorphism. Let $x\in\textbf{m}_p$ and write $h(x) = c+tf(t)$ for some $c\in k$ and $f(t)\in k[t]$. Then $h(x-c)\in (t)$. Suppose $c\neq 0$, so $x-c\notin\textbf{m}_p$ and $x-c$ is a unit. Since the image of a unit is a unit, $h(x-c)$ is a unit. This is a contradiction, so $c=0$ and $h$ is local. By \ref{cor:functor} there is a bijection between all such $h$ and all ring homomorphisms $\phi:\HS^n_{\Oc_{X,p}}\rightarrow k$. By the construction of the bijection in \ref{cor:functor},
$\phi(d_i s)=0$ for $i=0,\ldots,n$ and every $\phi$. Note that every  morphism $\HS^n_{\Oc_{X,p}/k}/\textbf{m}_p \HS^n_{\Oc_{X,p}/k}\rightarrow k$
lifts to such a $\phi$ via the quotient map $\psi: \HS^n_{\Oc_{X,p}/k}\rightarrow\HS^n_{\Oc_{X,p}/k}/\textbf{m}_p\HS^n_{\Oc_{X,p}/k}$ and thus the regular functions $\psi(d_i s)$ vanish at every $k$-point of $\spec \left(\HS^n_{\Oc_{X,p}/k}/\textbf{m}_p\HS^n_{\Oc_{X,p}/k}\right)$ for $i=0,\ldots,n$. However, $\spec\left(\HS^n_{\Oc_{X,p}/k}/\textbf{m}_p\HS^n_{\Oc_{X,p}/k}\right) \cong \A^{nq}_k$ by Proposition \ref{prop:smoothcover}. The only regular function vanishing everywhere on $\A^{nq}_k$ is the zero function, so $\{d_i s\}_{i=0}^n\subset \textbf{m}_p \HS^n_{\Oc_{X,p}/k}$.

For the opposite direction suppose $\mult_p D = w < n+1$. By hypothesis $\Oc_{X,p}$ is a regular local ring and $\textbf{m}_p = (x_1,\ldots,x_q)$ where $q$ is the dimension of $\Oc_{X,p}$. Write $s$ as $s_1+s_2$ where $s_1 = \sum \alpha_i u_i$ with $\alpha_i\in k[x_1,\ldots,x_q]$ monomials of degree $w$, $u_i\in\Oc_{X,p}^*$ and $s_2\in \textbf{m}_p^{w+1}$. Let $\bar{u}_i\in k$ be the image of $u_i$ in the quotient map $\Oc_{X,p}/\textbf{m}_p = k$. Since $k$ is algebraically closed there exists a $k$-point $r=(y_1,\ldots,y_q)\in \spec{k[x_1,\ldots,x_q]}$ such that $\sum \alpha_i(r) \bar{u}_i = c$ for some $c\in k^*$. By \ref{prop:loc} any morphism $h:\Oc_{X,p}\rightarrow k[[x_1,\ldots,x_q]]\rightarrow k[t]/t^{n+1}$ is determined freely by the images of $x_i$. Define $h(x_j) = ty_j$ for $j\in\{1,\ldots,q\}$ and note that  $h(u_i)=\bar{u}_i+tu_i'$ for some $u_i'\in k[t]/t^{n+1}$. Combining everything gives
\[
\begin{split}
h(s_1) &= \sum h(\alpha_i)h(u_i)\\
&= \sum h(\alpha_i)\bar{u}_i + t \sum h(\alpha_i)u_i'\\
&= t^w\sum \alpha_i(r)\bar{u}_i + t \sum h(\alpha_i)u_i'\\
&= ct^w + t \sum h(\alpha_i)u_i'
\end{split}
\]
However, $h(\alpha_i)$ is divisible by $t^w$, so the second term is divisible by $t^{w+1}$ and $h(s_1)\neq 0$. Thus $h(s) \neq 0$. By \ref{cor:functor} there exists a homomorphism $\phi_h:\HS^n_{\Oc_{X,p}/k}\rightarrow k$ associated to $h$ such that $\phi_h(d_w s)=c$. Let $\pi:\Oc_{X,p}\rightarrow\HS^n_{\Oc_{X,p}/k}$ be the natural inclusion map. Since $h(\textbf{m}_p) \subset (t)$, the composition $\Oc_{X,p}\stackrel{h}{\rightarrow} k[t]/t^{n+1}\rightarrow k[t]/t = k$ is just the map $\Oc_{X,p}\rightarrow \Oc_{X,p}/\textbf{m}_p =k$, so $\phi_h(\pi(\textbf{m}_p)) =0$ by construction of $\phi_h$. Thus $\phi_h$ must factor through the quotient map
\[
\xymatrix{
\HS^n_{\Oc_{X,p}/k} \ar[d]\ar[r]^{\phi_h}& k\\
\HS^n_{\Oc_{X,p}/k}/\textbf{m}_p\HS^n_{\Oc_{X,p}/k}\ar[ur]}
\]
However, $\phi_h(d_w s) \neq 0$, so $d_w s\notin \textbf{m}_p\HS^n_{\Oc_{X,p}/k}$.
\end{proof}

\begin{mydef} \label{def:globaljet}
Let $X$ be a log scheme over $Y$ and define the \textit{global log jet sheaf} to be the subsheaf $\HSG^n_{X/Y}(M_X/M_Y)\subset \HS^n_{X/Y}(M_X/M_Y)$ of $\Oc_X$-algebras defined locally by
\[
\Gamma(V,\HSG^n_{X/Y}(M_X/M_Y)) := \Oc_X(V)[d^i\omega_v]_{\substack{ \omega_v\in\Gamma(X,\Omega_{X/Y}(M_X/M_Y))\\i=0,\ldots,n-1}} \subset \Gamma(V,\HS^n_{X/Y}(M_X/M_Y))
\]
for each open $V\subset X$. Additionally define the \textit{global log jet space} to be the scheme
\[
\JG^n_{X/Y}(M_X/M_Y) := \sheafspec\HSG^n_{X/Y}(M_X/M_Y)
\]
over $X$.
\end{mydef}

\begin{mydef} \label{def:coords}
Let $X$ be a log scheme over $Y$, $D$ a Cartier divisor on $X$ and $U\subset X$ an open subset. We say $D$ \textit{has global coordinates in $M_X/M_Y$} on $U$, if there exists an open cover $U = \bigcup U_i$ such that $D$ can represented by $s_i\in \Oc_X(U_i)$ where
\[
d_n s_i\in\Gamma(U_i, \HSG^n_{X/Y}(M_X/M_Y))
 \]
for all $i$ and $n\geq 0$.
\end{mydef}

\begin{thm}[Main Theorem]\label{thm:main}
Let $\bar{X},\bar{Y}$ be irreducible proper coherent log schemes over an algebraically closed field $k$. Let $X\subset \bar{X}$ and $Y\subset \bar{Y}$ be the complement of the support of the log structure in $\bar{X}$ and $\bar{Y}$ respectively, and $\bar{D}$ an effective Cartier divisor on $\bar{X}$ with global coordinates in $M_{\bar{X}/k^*}$ on $U\subset X$. Then there exists some integer $N$ such that for all $f\in \Hom_k(\bar{Y},\bar{X})$ in the category of log schemes either $f(\bar{Y})\subset \textrm{Supp}(\bar{D})$ or $\mult_p f^*(\bar{D})\leq N$ for all nonsingular $p\in f^{-1}(U\cap\bar{D})\cap Y$.
\end{thm}

\begin{proof}
For notational reasons define
\[
\JG^n_{\bar{X}/k}(M_{\bar{X}}/k^*)|_U := \JG^n_{\bar{X}/k}(M_{\bar{X}}/k^*)\times_{\bar{X}}U
\]
for each $n\geq 0$ to be the restriction of the global jet space over the open set $U\subset X$, where the fiber products are taken in the category of schemes. Let $V\subset Y$ be the open subset of nonsingular points. Since $Y$ is the complement of the support of the log structure, the restriction of the log structure to $Y$ is trivial. Therefore define $M_V$ to be the trivial log structure as well.

By Proposition \ref{coherentdiff} the sheaves $\Omega_{\bar{X}/\spec k}(M_{\bar{X}}/k^*)$ and $\Omega_{\bar{Y}/\spec k}(M_{\bar{Y}}/k^*)$ are coherent. Since $\bar{X}$ and $\bar{Y}$ are proper,
the vector spaces $\Gamma(\bar{X},\Omega_{\bar{X}/\spec k}(M_{\bar{X}}/k^*))$ and $\Gamma(\bar{Y},\Omega_{\bar{Y}/\spec k}(M_{\bar{Y}}/k^*))$ are finite dimensional by Serre's Finiteness Theorem \cite{serre}. Let $\{\omega_1,\ldots,\omega_r\}$ and $\{\lambda_1,\ldots,\lambda_u\}$ be their bases, respectively. For each $n\geq 0$ define the schemes
\[
L_n := \spec k[x_{vj}]_{v=1,\ldots, r; j=1,\ldots, u}\times_k\JG^n_{\bar{X}/k}(M_{\bar{X}}/k^*)|_U \times_k \J^n_{V/k}(M_V/k^*)
\]
and ideal sheaves
\[
\mathscr{I}_n := \left(1\otimes d^m(\omega_v)\otimes 1-\sum_{j=1}^u x_{vj}\otimes 1\otimes d^m(\lambda_j)\right)_{m=0,\ldots,n-1;v=1,\ldots,r}
\]
Denote by $I_n\subset L_n$ the corresponding closed subscheme. By Definition \ref{def:globaljet} the structure sheaf of $\JG^n_{\bar{X}/k}(M_{\bar{X}}/k^*)|_U$ is locally generated over $\Oc_{\bar{X}}$ by the $d^m(\omega_v)$ for $m=0,\ldots,n-1$ and $v=1,\ldots,r$ giving a closed immersion
\[
j_n: I_n\rightarrow \A_k^{ru} \times_k U\times_k \J^n_{V/k}(M_V/k^*)
\]
for each $n$.

Any morphism $f:\bar{Y}\rightarrow \bar{X}$ of log schemes induces a $k$-linear map on the global log differential forms
\[
f^*:\Gamma(\bar{X},\Omega_{\bar{X}/\spec k}(M_{\bar{X}}/k^*))\rightarrow\Gamma(\bar{Y},\Omega_{\bar{Y}/\spec k}(M_{\bar{Y}}/k^*)).
\]
Let $l(f)$ be the point in $\A^{ru}_k$ where the $x_{vj}$ are determined by the above linear map. Suppose there is some $p\in V$ such that $f(p)\in U$. Then the triple $(l(f),f(p),p)$ can be viewed as a point in $L_0$. In order to establish some sort of bound on intersection multiplicity we need to work over the fibers of the log jet space, but fortunately $f$ induces a morphism on fibers
\[
f: \left(J^n_{V/k}(M_V/k^*)\right)_p\rightarrow \left(\JG^n_{\bar{X}/k}(M_{\bar{X}}/k^*)\right)_{f(p)}
\]
and the graph of this morphism
\[
\left(\JG^n_{\bar{X}/k}(M_{\bar{X}}/k^*)\right)_{f(p)}\times_{\left(\JG^n_{\bar{X}/k}(M_{\bar{X}}/k^*)\right)_{f(p)}} \left(J^n_{V/k}(M_V/k^*)\right)_p\cong \left(I_n\right)_{(l(f),f(p),p)}
\]
is isomorphic to the fiber of $I_n$ over the point $(l(f),f(p),p)\in L_0$ by construction. In particular
\begin{equation} \label{eq:isofibers}
\left(I_n\right)_{(l(f),f(p),p)} \cong \left(J^n_{V/k}(M_V/k^*)\right)_p
\end{equation}
provided $(l(f),f(p),p)$ is a triple coming from a morphism as described above. This isomorphism is compatible with the morphism $j_n$, i.e., $j_n$ induces the above isomorphism. The assumption that the triple come from an actual morphism is necessary as it was used in order to construct the graph.

Define a sheaf of ideals locally by
\[
\mathscr{P}_n = \left(1\otimes d_j s_i\otimes 1\right)_{j=0,\ldots,n}
\]
on the scheme $L_n$ where the $s_i$ come from Definition \ref{def:coords} and let $P_n$ be the associated closed subscheme.

Again let $(l(f),f(p),p)\in L_0$ be a triple coming from a morphism of log schemes as described above and suppose $(I_n)_{(l(f),f(p),p)}\subset (P_n)_{(l(f),f(p),p)}$ where $f(p)\in U_i$. Then $d_js_i$ vanishes at every point in the fiber $(I_n)_{(l(f),f(p),p)}$ for $j=0,\ldots,n$. By the isomorphism coming from the graph of $f$ in Equation \ref{eq:isofibers} we must have $d_jf^*s_i$ vanishing at every point of the fiber $\left(J^n_{V/k}(M_V/k^*)\right)_p$ for $j=0,\ldots,n$. However, $p$ is a smooth point by assumption, so by Theorem \ref{thm:factorfibers}, $\mult_p f^*s_i\geq n+1$.

On the other hand suppose that $(I_n)_{(l(f),f(p),p)}\not\subset (P_n)_{(l(f),f(p),p)}$. Then by the isomorphism in Equation \ref{eq:isofibers} there exists some point of $\left(J^n_{V/k}(M_V/k^*)\right)_p$ such that some $d_j f^*s_i$ does not vanish. Therefore $d_j f^*s_i\notin \textbf{m}_p \HS^n_{\Oc_{Y,p}/k}$ and by Theorem \ref{thm:factorfibers}, $\mult_p f^*s_i<n+1$. Thus
$(I_n)_{(l(f),f(p),p)} \subseteq (P_n)_{(l(f),f(p),p)}$ if and only if
$\mult_p f^*D \ge n+1$.

Finally define a decreasing sequence of closed subsets by
\[
H_n := L_0\setminus\pi_n\left(\A^{ru}_k\times_k U\times_k J^n_{V/k}(M_V/k^*)\setminus j_n(I_n\cap P_n)\right)
\]
where
\[
\pi_n: \A^{ru}_k\times_k U\times_k J^n_{V/k}(M_V/k^*)\rightarrow \A^{ru}_k\times_k U\times_k V
\]
is the projection map. The $H_n$ are decreasing because the conditions to be an element are stronger for successive $n$. To see that they are closed the expression need only be unraveled step by step. Indeed $j_n$ is a closed immersion, its complement is open, $\pi_n$ is flat and in particular open so the complement of its image is closed.

Let $(l(f),f(p),p)\in L_0$ be a point coming from a morphism of log schemes. Then by construction $(l(f),f(p),p)\in H_n$ if and only if $(I_n)_{(l(f),f(p),p)}\subset (P_n)_{(l(f),f(p),p)}$ which by previous observation happens if and only if $\mult_p f^*D\geq n+1$. Since $L_0$ is noetherian the $H_n$ must eventually stabilize for some $N-1$. Therefore $\mult_p f^*D\leq N$ or $\mult_p f^*D = \infty$ in which case $f(\bar{Y})\subset \textrm{Supp}(\bar{D})$.
\end{proof}

\begin{rem}
Theorem \ref{thm:main} does not include any assumption on the characteristic of the field. This might seem surprising since most ABC type estimates require substantial modification in order to work in positive characteristic. However, the hypothesis of Theorem \ref{thm:main} will only be met in positive characteristic if $D=0$ or if $U$ is disjoint from the support of $D$.
\end{rem}

The results can also be easily interpreted in affine $n$ space by the following proposition.

\begin{prop}\label{prop:projspace}
Let $V=\mathbb{P}_k^n$ with coordinates $[x_0:\cdots:x_n]$, and  $f(x_0,\ldots,x_n)$ a homogeneous polynomial of degree $m>0$. Give $V$ the log structure defined by $x_n$ and $f$. Then $\partial_1 f(\frac{x_0}{x_n},\frac{x_1}{x_n},\ldots, 1)$  is the restriction of a global differential log form to the complement of the hyperplane $x_n=0$.
\end{prop}

\begin{proof}
On each $V_i = \spec k[\frac{x_0}{x_i},\ldots,\frac{x_n}{x_i}]$ the log structure is generated by $\frac{x_n}{x_i}$ and $f(\frac{x_0}{x_i},\ldots, \frac{x_n}{x_i})$. Define
\[
\omega_i = \partial_1 f\left(\frac{x_0}{x_i},\ldots, \frac{x_n}{x_i}\right)-m\partial_1\left(\frac{x_n}{x_i}\right),
\]
which is a well defined log 1-form. On $V_i\cap V_j$ we have
\[
\begin{split}
\omega_i
&= \partial_1\left(f\left(\frac{x_0}{x_j},\frac{x_1}{x_j},\ldots,\frac{x_n}{x_j}\right)\left(\frac{x_j}{x_i}\right)^m\right)-m\partial_1\left(\frac{x_n}{x_j}\frac{x_j}{x_i}\right)\\
&= \partial_1 f\left(\frac{x_0}{x_j},\frac{x_1}{x_j},\ldots, \frac{x_n}{x_j}\right) + m\partial_1\left(\frac{x_j}{x_i}\right) - \left(m\partial_1\left(\frac{x_n}{x_j}\right)+m\partial_1\left(\frac{x_j}{x_i}\right)\right)\\
&= \partial_1 f\left(\frac{x_0}{x_j},\frac{x_1}{x_j},\ldots, \frac{x_n}{x_j}\right)-m\partial_1 \left(\frac{x_n}{x_j}\right)\\
&= \omega_j.
\end{split}
\]
Since $i$ and $j$ were arbitrary and they agree on every intersection, there exists a global log form $\omega$ whose restriction to $V_i$ is $\omega_i$. In particular on $V_n$ we have that
\[
\omega_n = \partial_1 f\left(\frac{x_0}{x_n},\ldots, 1\right)
\]
is the restriction of a global log form.
\end{proof}

This allows for the main theorem to be applied in some instances without having to use log geometry, compactification or higher order log differential forms.

\begin{cor}
Let $\{f_i\}_{1=1,\ldots,m}$ be a finite set of regular functions on $\A^n_\C$ such that $\{df_i\}_{i=1,\ldots,m}$ generate the sheaf of differentials $\Omega^1_{\A^n_\C/\C}$. Let $D$ be any effective Cartier divisor on $\A^n_\C$. Then for any affine curve $C$ there exists some integer $N$ such that for every morphism $j: C\rightarrow \left(\A^n_\C\right)_{f_1\cdots f_m}$ either $\mult_p j^*D \leq N$ for every $p\in C$ or $j(C)\subset D$.
\end{cor}

\begin{proof}
Immediate from Theorem \ref{thm:main}, Proposition \ref{prop:projspace} and Proposition \ref{thm:gendiff}.
\end{proof}

\bibliographystyle{plain}
\bibliography{thesis}

\end{document}